\documentclass{amsart}
\usepackage[T1]{fontenc}
\usepackage[latin9]{inputenc}
\usepackage{amsmath}
\usepackage{amscd}
\usepackage{amssymb}
\usepackage{yfonts}
\usepackage{mathrsfs}
\usepackage{textcomp}
\usepackage{pictexwd,dcpic}
\usepackage{xfrac}
\usepackage{enumerate}
\usepackage{hyperref}
\usepackage{tikz}
\usetikzlibrary{cd}
\usepackage{graphicx}
\usepackage{anysize}
\usepackage{mdwlist}

\newtheorem{thm}{Theorem}[section]
\newtheorem{cor}[thm]{Corollary}
\newtheorem{lem}[thm]{Lemma}
\newtheorem{prop}[thm]{Proposition}
\theoremstyle{definition}
\newtheorem*{defi}{Definition}
\newtheorem{conj}{Conjecture}

\theoremstyle{remark}
\newtheorem{remark}{Remark}

\numberwithin{equation}{section}

\begin{document}

\title{A Schanuel Property for $j$}
\author{Sebastian Eterovi\'c}
\date{\today}

\address{Mathematical Institute -- University of Oxford, UK}

\email{eterovic@maths.ox.ac.uk}

\maketitle

\begin{abstract}
    I give a model-theoretic setting for the modular $j$ function and its derivatives. These structures, here called $j$-fields, provide an adequate setting for interpreting the Ax-Schanuel theorem for $j$ (Theorem $1.3$ of \cite{pila-tsimerman}). Following the ideas of \cite{bays-kirby-wilkie} and \cite{kirby} for exponential fields, I prove a generic transcendence property for the $j$ function.
\end{abstract}

\section{Introduction}

The aim of this work is to prove a generic transcendence property for the $j$ function in the spirit of Schanuel's conjecture (see conjecture (S) of \cite{ax}). Such a result is proven in \cite{bays-kirby-wilkie} for the exponential function in the context of exponential fields.
\begin{thm}[see Theorem $1.2$ of \cite{bays-kirby-wilkie}]
\label{thm:exp}
Let $F$ be any exponential field, let $\lambda\in F$ be exponentially transcendental, and let $\overline{x}\in F^{n}$ be such that $\exp(\overline{x})$ is multiplicatively independent. Then:
\begin{equation*}
    \mathrm{t.d.}(\exp(\overline{x}),\exp(\lambda\overline{x})/\lambda)\geq n.
\end{equation*}
\end{thm}
The notation we use is as follows: for $A$ and $B$ subsets of a given field of characteristic zero, we write $\mathrm{t.d.}(A/B)$ to denote the transcendence degree of the field extension $\mathbb{Q}(A)/\mathbb{Q}(B)$. Given some elements $z_{1},\ldots,z_{n}$ in the field $F$, we write $\overline{z}$ to denote the tuple $(z_{1},\ldots,z_{n})$, and if $f$ is a function, then $f(\overline{z})$ denotes the tuple $(f(z_{1}),\ldots,f(z_{n}))$. 

Theorem \ref{thm:exp} relies heavily on Theorem $3$ of \cite{ax} (the so-called Ax-Schanuel theorem). In \cite{pila-tsimerman}, the authors prove an analogue of the Ax-Schanuel theorem for the $j$ function (see Theorem \ref{thm:ax-schanuel-j}), so it is natural to wonder if a strategy similar to the one used in \cite{bays-kirby-wilkie} would provide a transcendence result for $j$. This is what is done here.

To do this we will define ``$j$-fields'' in analogy to the way exponential fields are defined in \cite{kirby}. We will use some standard techniques of geometric model theory (like pregeometries) and o-minimality for our approach. A very short introduction to pregeometries is given in the Preliminaries. The proof of our main result (Theorem \ref{thm:main}) relies mostly on transcendence properties of fields, so the reader not familiar with the model-theoretic concepts may still follow the overall strategy. 

Traditionally, the $j$ function is understood as a modular function defined on the upper-half plane, but we will extend it to be defined on the upper and lower half planes so that $j:\mathbb{H}^{+}\cup\mathbb{H}^{-}\rightarrow\mathbb{C}$. Given $z\in\mathbb{H}^{-}$ we define $j(z):=\overline{j(\overline{z})}$, where $\overline{z}$ is the complex conjugate of $z$ (even though the symbol is the same, it should be clear from context when $\overline{z}$ denotes complex conjugation or a tuple of elements). We do this because the condition for an element $g\in \mathrm{GL}_{2}(\mathbb{Q})$ to preserve $\mathbb{H}^{+}\cup\mathbb{H}^{-}$ setwise is easier to check: namely $\det(g)\neq 0$. The importance of this will become apparent in the proof of Lemma \ref{lem:disjoint1}. If we only considered the upper-half plane, we would have to ask that $\det(g)>0$, which would require us to include an order relation in our $j$-field. Given the nature of the transcendence properties of $j$ that we are interested in, it does not matter if we extend the domain of $j$ in the way we have done. 

The term \emph{$j$-generic} used in the next theorem is analogous to that of \emph{exponential transcendence}; we will give a precise definition in Section \ref{sec:jder}. The full statement of Theorem \ref{thm:main} is rather technical and works on all $j$-fields, but one immediate consequence to a concrete setting is:
\begin{thm}[see Corollary \ref{cor:weak}]
\label{thm:intro}
Let $\tau\in\mathbb{R}$ be $j$-generic. Suppose $z_{1},\ldots,z_{n}\in \mathbb{H}^{+}\cup\mathbb{H}^{-}$ and $g\in\mathrm{GL}_{2}(\mathbb{Q}(\tau))$ are such that $z_{1},\ldots,z_{n},gz_{1},\ldots,gz_{n}$ are in different $\mathrm{GL}_{2}(\mathbb{Q})$-orbits (pairwise). Then:
\begin{equation*}
    \mathrm{t.d.}\left(j(\overline{z}),j'(\overline{z}),j''(\overline{z}),j(g\overline{z}),j'(g\overline{z}),j''(g\overline{z})/\tau\right)\geq 3n.
\end{equation*}
\end{thm}
Loosely, this should be understood as: if the entries of $g$ are sufficiently generic (transcendental with respect to $j$), then we obtain the above transcendence result for $j$ and its derivatives. Using o-minimality, we prove that all but countably many values for $\tau$ can be used (see Remark \ref{rem:countable}), although, just as we do not have any explicit examples of exponentially transcendental numbers, no explicit $j$-generic numbers are known.

As will be shown, $j$-fields prove to be an adequate setting in which to interpret the Ax-Schanuel theorem for $j$ (see Proposition \ref{prop:as}), that is to say, we can define certain pregeometries (here named $\mathrm{gcl}$ and $j\mathrm{cl}$) on $j$-fields that encode the contents of this result.

The paper is organized as follows. First we define $j$-fields axiomatically so as to mimic the behaviour of the $j$ function on $\mathbb{C}$. Then we study the family of pregeometries $\mathrm{gcl}_{F}$ obtained from the action of a subgroup of $\mathrm{GL}_{2}(F)$ on a field $K$, where $F$ is a subfield of $K$. We also introduce the notion of ``geodesical disjointness'', which is a concept that emulates the idea of two fields being ``linearly disjoint'' over a third. Lemmas \ref{lem:disjoint1} and \ref{lem:disjoint2} are crucial for obtaining our main theorem. 

After that, comes a section on $j$-derivations, which are derivations which respect the $j$ function and its derivatives. We define the operator $j\mathrm{cl}$ and prove that it is a pregeometry. Using standard techniques of o-minimality, we also show how to construct non-trivial $j$-derivations on $\mathbb{C}$. See \cite{macpherson} for a quick introduction to o-minimality.

The last section contains the transcendence results, the most important being Theorem \ref{thm:main} which generalises Theorem \ref{thm:intro}. As stated above, the strategy of the proofs is inspired by \cite{bays-kirby-wilkie}. We also analyse a modular version of Schanuel's conjecture and prove that, like in the case of the original Schanuel's conjecture, in $\mathbb{C}$ there are at most countably many essential counterexamples of this conjecture. At the end of this section we present some open problems. As is explained there, because of the abstract nature of derivations, it is not easy to decide when an element does not belong to the $j\mathrm{cl}$-closure of a given set, and a more explicit definition of $j\mathrm{cl}$ would be helpful. In the case of exponential fields, one can do this through Khovanskii systems (see Proposition $7.1$ of \cite{kirby}). However, the strategy used there relies on very specific properties of the exponential function, and so does not work for $j\mathrm{cl}$. 

\section{Preliminaries}

\subsection{Pregeometries}

Let $X$ be a set. A function $\mathrm{cl}:\mathcal{P}(X)\rightarrow\mathcal{P}(X)$ (here $\mathcal{P}(X)$ denotes the power set of $X$) is called a \emph{pregeometry} on $X$ if it satisfies the following properties for every $A,B\in\mathcal{P}(X)$:
\begin{enumerate}[(a)]
\item $A\subseteq\mathrm{cl}(A)$,
\item If $A\subseteq B$, then $\mathrm{cl}(A)\subseteq\mathrm{cl}(B)$,
\item $\mathrm{cl}(\mathrm{cl}(A)) = \mathrm{cl}(A)$,
\item Finite character: if $a\in\mathrm{cl}(A)$, then there is a finite subset $A_{0}\subseteq A$ such that $a\in \mathrm{cl}(A_{0})$,
\item Exchange: if $a,b\in X$ are such that $a\in\mathrm{cl}(A\cup\left\{b\right\})$ and $a\notin\mathrm{cl}(A)$, then $b\in\mathrm{cl}(A\cup\left\{a\right\})$.
\end{enumerate}
For example, if $X$ is a vector space, then we can take $\mathrm{cl}(A)$ to be the linear span of $A$. If $X$ is a field, then we can take $\mathrm{cl}(A)$ to be the set of all elements of $X$ algebraic over $A$. 

A crucial aspect of pregeometries is that they allow us to have well-defined notions of independence and dimension. A set $A\subseteq X$ is \emph{$\mathrm{cl}$-independent} if for every $a\in A$ we have that $a\notin\mathrm{cl}(A\setminus\left\{a\right\})$. The \emph{dimension} of a set $B\subseteq X$ with respect to $\mathrm{cl}$ is the cardinality of any $\mathrm{cl}$-independent set $A\subseteq B$ such that $\mathrm{cl}(A)=B$. For the basic properties of pregeometries, see Appendix C of \cite{tent-ziegler}.

\subsection{Review of Elliptic Curves and the Modular $j$-function}

We will now give a very brief review of complex elliptic curves to have enough context for the set up of the $j$-function. For the general theory of elliptic curves, and most of the results stated here, see e.g. \cite{silverman} and \cite{silverman2}.

A \emph{lattice} in $\mathbb{C}$ is a subgroup $\Lambda$ of the additive group of $\mathbb{C}$ generated by two elements $\omega_{1},\omega_{2}\in\mathbb{C}$ which are linearly independent over $\mathbb{R}$. This means that $\Lambda$ is isomorphic to $\mathbb{Z}\omega_{1}+\mathbb{Z}\omega_{2}$. The quotient space $E_{\Lambda} = \mathbb{C}/\Lambda$ has a natural complex structure induced by that of $\mathbb{C}$. In fact $E_{\Lambda}$ is a compact Riemann surface of genus $1$ and it can be realised as a complex plane projective curve by means of the Weierstrass $\wp$ associated to $\Lambda$ as follows. 

Given the lattice $\Lambda$, consider the Weierstrass $\wp$ function:
\begin{equation*}
    \wp(z) = \frac{1}{z^{2}} + \sum_{\lambda\in\Lambda, \lambda\neq 0}\left(\frac{1}{(z-\lambda)^{2}} - \frac{1}{\lambda^{2}}\right).
\end{equation*}
Using the Weierstrass $M$-test for convergence, one can see that $\wp(z)$ is a meromorphic function on $\mathbb{C}$ that is $\Lambda$-invariant. It also satisfies the differential equation: $\wp'(z)^{2} = 4\wp(z)^{3} - g_{2}\wp(z) - g_{3}$,
where:
\begin{equation*}
    g_{2} = 60\sum_{\lambda\in\Lambda,\lambda\neq 0}\frac{1}{\lambda^{4}},\qquad g_{3} = 140\sum_{\lambda\in\Lambda,\lambda\neq 0}\frac{1}{\lambda^{6}}.
\end{equation*}
So, the map $[z]\mapsto[\wp(z) : \wp'(z) : 1]$ defines an isomorphism between $E_{\Lambda}$ and the plane projective curve $Y^{2}Z = 4X^{3} - g_{2}XZ^{2}-g_{3}Z^{3}$. 

This construction motivates the following definition. An \emph{elliptic curve} is a plane projective curve given by an equation of the form:
\begin{equation}
\label{eq:elliptic}
    Y^{2}Z = 4X^{3} - aXZ^{2} - bZ^{3}, 
\end{equation}
such that the discriminant $\Delta := a^{3}-27b^{2}\neq 0$. Any lattice $\Lambda$ in $\mathbb{C}$ produces an elliptic curve $E_{\Lambda}$ by the above construction. In fact, all elliptic curves can be obtained in this way. Also notice that the analytic construction we gave of $E_{\Lambda}$ (as a quotient space $\mathbb{C}/\Lambda$) immediately shows that $E_{\Lambda}$ has a group structure (as a quotient group). But in fact the group operation is a rational map, and so it also compatible with the algebraic structure of $E_{\Lambda}$ (as a plane projective curve). For this reason we will use the name $E_{\Lambda}$ to refer both to the analytic and algebraic aspects of the elliptic curve. 

If $E$ is an elliptic curve defined by equation (\ref{eq:elliptic}), define the \emph{$j$-invariant of $E$} as:
\begin{equation*}
    j(E) := 1728\frac{a^{3}}{\Delta}.
\end{equation*}
A \emph{morphism of elliptic curves} is a morphism of algebraic varieties that is also a group homomorphism. For any elliptic curve $E'$ we have that $E$ is isomorphic to $E'$ if and only if $j(E) = j(E')$. Furthermore, given two lattices, $\Lambda$ and $\Lambda'$, $E_{\Lambda}$ is isomorphic to $E_{\Lambda'}$ if and only if there is $c\in\mathbb{C}^{\ast}$ such that $c\Lambda = \Lambda'$. So, all isomorphism classes of elliptic curves are obtained by considering only lattices of the form $\mathbb{Z} + \mathbb{Z}\tau$, where $\tau\in\mathbb{H}$. We can now formulate the definition we have been waiting for. 
\begin{defi}
The \emph{$j$-function} is the holomorphic map $j:\mathbb{H}\rightarrow\mathbb{C}$ given by $j(\tau) = j(E_{\tau})$, where $E_{\tau}$ is the elliptic curve defined by the lattice $\Lambda(\tau):=\mathbb{Z} + \mathbb{Z}\tau$. As was said in the introduction, we will extend this function to be defined on the upper and lower half-planes, so that $j:\mathbb{H}^{+}\cup\mathbb{H}^{-}\rightarrow\mathbb{C}$, where $j(\tau)$ for $\tau\in\mathbb{H}^{-}$ is defined as $j(\tau):=\overline{j(\overline{\tau})}$, where $\overline{z}$ denotes the complex conjugate of $z$. 
\end{defi}
We will now review a few properties of this function. As it can be easily verified, given $\tau, \tau'\in\mathbb{H}$, we have that $\Lambda(\tau) = \Lambda(\tau')$ if and only if there is $g\in\mathrm{SL}_{2}(\mathbb{Z})$ such that:
\begin{equation*}
    \frac{a\tau+b}{c\tau+d} = \tau',\quad \mathrm{where }\, \,g=\left(\begin{matrix}
    a & b\\ 
    c & d
    \end{matrix}\right).
\end{equation*}
This means that $j$ is invariant under the action of $\mathrm{SL}_{2}(\mathbb{Z})$. In fact, a lot more is known. The first thing is to note that $\mathrm{SL}_{2}(\mathbb{R})$ acts on $\mathbb{H}$ through M\"obius transformations, in fact $\mathrm{SL}_{2}(\mathbb{R})$ is the group of automorphisms of $\mathbb{H}$. Given $g\in \mathrm{GL}_{2}^{+}(\mathbb{Q})$ (the $+$ denotes positive determinant), there is a unique positive integer $N(g)\in\mathbb{N}$ such that $N(g)g\in\mathrm{GL}_{2}(\mathbb{Z})$ and the entries of $N(g)g$ are relatively prime. There exists a family of polynomials $\left\{\Phi_{N}\right\}_{N\geq 1}\subseteq\mathbb{Z}[X,Y]$, usually referred to as the \emph{modular polynomials}, that satisfy $\Phi_{N}(j(x), j(y))=0$ if and only if there is $g\in\mathrm{GL}_{2}^{+}(\mathbb{Q})$ such that $\det(N(g)g) = N$ and $x=gy$. The first modular polynomial is $\Phi_{1}(X,Y) = X-Y$, and $\Phi_{N}(X,Y)$ is symmetric for $N\geq 2$. For a precise construction of these polynomials, see e.g. section 6.1 of \cite{zagier}.

So the modular polynomials give us algebraic properties of the $j$-function. On the analytic side, $j$ satisfies the following differential equation of order $3$:
\begin{equation}
\label{eq:diffj}
    \frac{j'''}{j'} - \frac{3}{2}\left(\frac{j''}{j'}\right)^{2} + \left(\frac{j^{2}-1968j+2654208}{2j^{2}(j-1728)^{2}}\right)(j')^2 = 0.
\end{equation}
This can be proven using the definition of $j$ given in page 22 of \cite{zagier} plus Proposition 15 of \cite{zagier}. So, for purposes of transcendence of $j$ and its derivatives, it suffices to consider only $j$, $j'$ and $j''$. Thanks to the Modular Ax-Lindemann-Weierstrass (ALW) theorem below, we can also say that the only algebraic relations that the functions $j(z)$ and $j(gz)$ can satisfy are those given by the modular polynomials. On an algebraic function field $K$ of characteristic zero, $x_{1},\ldots,x_{n}\in K$ are said to be \emph{geodesically independent} if the $x_{i}$ are nonconstant and there are no relations of the form $x_{i} = gx_{k}$, where $i\neq k$ and $g\in\mathrm{GL}_{2}^{+}(\mathbb{Q})$. 

\begin{thm}[Modular ALW, theorem $1.1$ of \cite{pila}]
Let $\mathbb{C}(W)$ be an algebraic function field, where $W\subseteq\mathbb{C}^{n}$ is an irreducible algebraic variety. Suppose that $a_{1},\ldots,a_{n}\in\mathbb{C}(W)$ take values in $\mathbb{H}$ at some $P\in W$. If $a_{1},\ldots,a_{n}$ are geodesically independent, then the $3n$ functions:
\begin{equation*}
    j(a_{1}),\ldots,j(a_{n}),j'(a_{1}),\ldots,j'(a_{n}),j''(a_{1}),\ldots,j''(a_{n})
\end{equation*}
(considered as functions on $W$ locally near $P$) are algebraically independent over $\mathbb{C}(W)$. 
\end{thm}
The more general Ax-Schanuel theorem for $j$ is stated later (see Theorem \ref{thm:ax-schanuel-j}).

\section{$j$-Fields}

In this section we define $j$-fields. Given a field $K$, for any subfield $F$ of $K$, there is a natural action of $\mathrm{GL}_{2}(F)$ on $\mathbb{P}^{1}(K) = K\cup\left\{\infty\right\}$, given by:
\begin{equation*}
    gx = \frac{ax+b}{cx+d},
\end{equation*} where $g\in\mathrm{GL}_{2}(F)$ is represented by $g=\left(\begin{matrix}
a & b\\
c & d
\end{matrix}\right)$. Whenever we say that $\mathrm{GL}_{2}(F)$ acts on $K$, it will be in this manner. Throughout, let $G = \mathrm{GL}_{2}(\mathbb{Q})$. 

\begin{defi}
A \emph{$j$-field} is a two-sorted structue $\left<\mathbb{K}, \mathbb{D}, \alpha,j,j',j'',j'''\right>$, where:
\begin{itemize}
    \item $\mathbb{K} = \left<K, +, \cdot, 0, 1\right>$ is a field of characteristic zero,
    \item $\mathbb{D} = \left<D, \left(g\right)_{g\in G}\right>$ is a $G$-set (i.e. a set with an action of $G$),
\end{itemize} and $\alpha,j,j',j'',j''':D\rightarrow K$ are maps that satisfy:
\begin{enumerate}
\item $\alpha$ is injective.  
\item\label{ax:alpha} For every $z\in D$  and $g\in G$, $\alpha(gz) = g\alpha(z)$.  
\item\label{ax:diffj} For every $z\in D$, 
\begin{equation*}
    (j(z)\neq 0\wedge j(z)\neq1728 \wedge j'(z)\neq 0)\implies\daleth(j(z),j'(z),j''(z),j'''(z))=0,
\end{equation*}
where $\daleth\in\mathbb{Q}(X,Y,Z,W)$ is given by:
\begin{equation}
\label{eq:j}
\daleth(X,Y,Z,W):=\frac{W}{Y} - \frac{3}{2}\left(\frac{Z}{Y}\right)^{2} + \left(\frac{X^{2}-1968X+2654208}{2X^{2}(X-1728)^{2}}\right)Y^2.
\end{equation}
Thus $\daleth(j(z),j'(z),j''(z),j'''(z))=0$ corresponds to equation (\ref{eq:diffj}).
\item\label{ax:modpoly1} The axiom scheme: for every $z_{1},z_{2}\in D$, if $z_{1} = gz_{2}$, then $\Phi_{N}(j(z_{1}),j(z_{2}))=0$ (note that given $g$ we can obtain the value of $N$ as $N=\det(N(g)g)$, so this axiom scheme is first-order expressible). We also include here the expressions that can be obtained by deriving up to $3$-times modular relations. This means the following. Choosing $g$ and $N$ as before, we have that for every $z\in D$, $\Phi_{N}(j(z),j(gz))=0$. If we interpret this expression in $\mathbb{C}$ and derive it with respect to $z$, then we get: 
\begin{equation}
\label{eq:j'}
    \Phi_{N1}(j(z),j(gz))j'(z) + \Phi_{N2}(j(z),j(gz))j'(gz)\frac{ad-bc}{(cz+d)^{2}} = 0,
\end{equation}
where $\Phi_{N1}$ and $\Phi_{N2}$ are the derivatives of $\Phi_{N}(X,Y)$ with respect to the variables $X$ and $Y$ respectively. So, for each $g\in G$, we also include the axiom that says: if $z_{1}=gz_{2}$, then:
\begin{equation*}
    \Phi_{N1}(j(z_{1}),j(z_{2}))j'(z_{1}) + \Phi_{N2}(j(z_{1}),j(z_{2}))j'(z_{2})\frac{ad-bc}{(c\alpha(z_{1})+d)^{2}} =0. 
\end{equation*}
Deriving equation (\ref{eq:j'}) again with respect to $z$, we get another equation, this time involving $j''(z)$ and $j''(gz)$, that we restate as a first-order axiom as we just did with equation (\ref{eq:j'}). Deriving (\ref{eq:j'}) twice with respect to $z$, we get an equation involving $j'''(z)$ and $j'''(gz)$, which we likewise restate as an axiom. 
\item\label{ax:modpoly} The axiom scheme (one $\mathscr{L}_{\omega_{1},\omega}$-statement for each $N\in\mathbb{N}$): for all $z_{1},z_{2}\in D$, 
\begin{equation*}
    \Phi_{N}(j(z_{1}), j(z_{2})) = 0\implies \bigvee_{g\in G, \det(N(g)g)=N} (gz_{2} = z_{1}),
\end{equation*}
so this axiom is a converse to part of axiom \ref{ax:modpoly1}.
\item\label{ax:kernel} The $\mathscr{L}_{\omega_{1},\omega}$-statement: let $X\subseteq G$ be the set of non-scalar matrices. Then for every $z\in D$, 
\begin{equation*}
    (j(z)=0\vee j(z)=1728\vee j'(z)=0)\implies \bigvee_{g\in X}(gz=z). 
\end{equation*}
Note that the values chosen are the same as those used in axiom \ref{ax:diffj}, and correspond to the points $z$ where the expression $\daleth(j(z),j'(z),j''(z),j'''(z))$ is not defined.
\end{enumerate}
\end{defi}

The name ``$j$-field'' and part of this definition come from work of Jonathan Kirby and Adam Harris privately communicated to the author.

Note that from axiom \ref{ax:alpha} we get that $\alpha(D)\cap\mathbb{Q}=\emptyset$ because the image of $\alpha$ is contained in $K$ and for every element $x\in\mathbb{Q}$ there is $g\in G$ such that $gx=\infty$. Also, the axioms for $\alpha$ say that $D$ is embedded in $K$, and this allows us to identify $D$ with $\alpha(D)$. Therefore we will avoid mentioning $\alpha$ in the future. For the case of the modular $j$-function, we take $\alpha:\mathbb{H}^{+}\cup\mathbb{H}^{-}\rightarrow\mathbb{C}$ to be the inclusion. 


\begin{remark}
The set of axioms we have given does not define a complete theory, we have only given the axioms that will be needed for our transcendence results. The most natural way to define a complete set of axioms would be to take the full theory of the $j$-field $(\mathbb{C},\mathbb{H}^{+}\cup\mathbb{H}^{-},j,j',j'',j''')$. Call this theory $\mathrm{Th}_{j}$. From a model-theoretic perspective, we should only focus on models of $\mathrm{Th}_{j}$, and maybe the name ``$j$-fields'' should be reserved for these structures. Even if the theory $\mathrm{Th}_{j}$ is not the subject of this paper, we will point out a couple of things.
\begin{enumerate}[(a)]
\item If $(K,D)\models\mathrm{Th}_{j}$, then any point of $D$ which is fixed by some non-scalar $g\in G$ is $\emptyset$-definable (using terminology we will introduce shortly, this says that the special points are $\emptyset$-definable), and so their images under $j$ are also $\emptyset$-definable. As it turns out, these points are the only points $x$ such that both $x$ and $j(x)$ are algebraic over $\mathbb{Q}$ (this is a famous theorem of Schneider, see \cite{schneider}). Using some known values of $j$ one can prove that numbers such as $\sqrt{2}$, $\sqrt{3}$ and $\sqrt{5}$ are definable in $K$.  
\item Note that $\mathbb{C}\setminus(\mathbb{H}^{+}\cup\mathbb{H}^{-})$ defines the real closed field $\mathbb{R}$. So if $(K,D)\models\mathrm{Th}_{j}$, then $K\setminus D$ is a real closed field. Note that it is not obvious how to associate to a given real closed field a corresponding $j$ function on its algebraic closure. In \cite{peterzil-starchenko} it is proven that the analytic construction of $j$ shown in the Preliminaries using genus one tori need not work in general. Also, given that real closed fields are not categorical nor $\omega$-stable, it should not be expected that $j$-fields have these properties. 
\end{enumerate}
Because of this last point, $\mathrm{Th}_{j}$ is not a very nice theory to work with from a model-theory viewpoint. But our aim is to show that, even so, one can obtain interesting number-theoretic results. 
\end{remark}

\begin{defi}
Let $F$ be a subfield of $K$. We define $G^{F}$ to be the subgroup of $\mathrm{GL}_{2}(F)$ defined as $G^{F} := \left\{g\in\mathrm{GL}_{2}(F) : gD\subseteq D\right\}$. Note that $G^{\mathbb{Q}}=G$. We say that $F$ is an \emph{active subfield} of $K$ if $G^{F} = \mathrm{GL}_{2}(F)$. For example, in the case of $K=\mathbb{C}$, every subfield of $\mathbb{R}$ is an active subfield of $K$.

A \emph{morphism of $j$-fields} $\sigma:(K_{1},D_{1})\rightarrow(K_{2},D_{2})$ is a field morphism $\sigma:K_{1}\rightarrow K_{2}$ such that the map $\sigma\circ\alpha_{1}$ is contained in the image of $\alpha_{2}$ and  $\widetilde{\sigma}:=\alpha_{2}^{-1}\circ\sigma\circ\alpha_{1}:D_{1}\rightarrow D_{2}$ satisfies: for every $f\in\left\{j,j',j'',j'''\right\}$ and for every $z\in D$ we have $\sigma(f_{1}(z)) = f_{2}(\widetilde{\sigma}(z))$. Note that field morphisms respect the action of the group $G^{F}$ on $K$, so we clearly have $\widetilde{\sigma}(gz)=g\widetilde{\sigma}(z)$ for every $g\in G$ and $z\in D$. 

Let $(K,D)$ be a $j$-field and let $F$ be a subfield of $K$. We say that $F$ is a \emph{$j$-subfield} of $K$ if the pair $(F,D_{F})$, where $D_{F}=\alpha^{-1}(D\cap F)$, is a $j$-field with the functions of $(K,D)$ restricted to $(F,D_{F})$. Note that if $(K_{i})_{i\in I}$ is a family of $j$-subfields of $K$, then $\bigcap_{i\in I}K_{i}$ is a $j$-subfield of $K$.
\end{defi}

\section{Geodesic Closures}

In this section we study the pregeometry $\mathrm{gcl}_{F}$ which is defined by the action of $G^{F}$ on $K$, for every subfield $F$ of $K$.

\begin{defi}
Let $A$ be a subset of $K$ and $F$ a subfield of $K$. We define the \emph{$F$-geodesic closure of $A$}, denoted $\mathrm{gcl}_{F}(A)$, as the set of $x\in K$ such that there exist $a\in A$ and $g\in G^{F}$ such that $x = ga$ (which, save for the exclusion of the point at infinity, is the union of the $G^{F}$-orbits of points in $A$). When $F=\mathbb{Q}$ we will simply write $\mathrm{gcl}(A)$. 

It is straightforward to check that for every subfield $F$ of $K$ the operator $\mathrm{gcl}_{F}$ is a pregeometry of trivial type. So given $A,B\subseteq K$ let $\dim^{g}_{F}(A/B)$ be the dimension defined by the pregeometry $\mathrm{gcl}_{F}$, i.e. the number of distinct orbits of elements in $A$ that do not contain elements of $B$. If $B=\emptyset$ then we write simply $\dim^{g}_{F}(A)$. 

Let $E$ be a subfield of $K$. A point $x\in K$ is \emph{$E$-special} (or \emph{special over $E$}) if there is a non-scalar $g\in G^{E}$ such that $gx=x$ (when $E=\mathbb{Q}$ we will simply say that $x$ is \emph{special}). 
\end{defi}

It is also easy to see that if $E$ is an active subfield of $K$, then $\dim^{g}_{E}(E)=1$. Note that if $L\subseteq F\subseteq K$ are subfields of $K$, then for every $A\subseteq K$ we have that  $\dim_{F}^{g}(A)\leq\dim_{L}^{g}(A)$.

\subsection{Geodesic Disjointness}

The following definition of geodesic disjointness is analogous to the notion of linear disjointness (see e.g. Definition $3.1$ of \cite{bays-kirby-wilkie}). 

\begin{defi}
Let $E,F,L$ be subfields of $K$ such that $E\subseteq F\cap L$. We say that $F$ is \emph{geodesically disjoint from $L$ over $E$}, denoted $F\bot^{g}_{E}L$, if for every tuple $\overline{\ell}$ of elements of $L$ that is $\mathrm{gcl}_{E}$-independent is also $\mathrm{gcl}_{F}$-independent. Alternatively, $F\bot^{g}_{E}L$ if and only if for any tuple $\overline{\ell}$ from $L$, $\dim^{g}_{F}(\overline{\ell}) = \dim^{g}_{E}(\overline{\ell})$.
\end{defi}

\begin{lem}
\label{lem:disjoint1}
Let $E,L$ be subfields of $K$ such that $E\subseteq L$ and $E$ is active. Suppose that $\overline{x}$ is a finite tuple from $K$ which is algebraically independent over $L$ and that there are non-scalar elements in $G^{E(\overline{x})}\setminus G^{E}$. Then:
\begin{enumerate}[(a)]
    \item $E(\overline{x})\bot^{g}_{E}L$. 
    \item If $\ell\in L$ is special over $E(\overline{x})$, then it is special over $E$.
\end{enumerate}
\end{lem}
\begin{proof}
 Let $\ell_{1},\ell_{2}\in L$ (not necessarily distinct) and let $g\in G^{E(\overline{x})}\setminus G^{E}$ be non-scalar. Suppose that $\ell_{1} = g\ell_{2}$. First observe that either $\ell_{1},\ell_{2}\in E$ or $\ell_{1},\ell_{2}\in L\setminus E$. To see this, suppose first that $\ell_{2}\in E$. Then, as $\ell_{1}=g\ell_{2}$ and $g\in G^{E(\overline{x})}$, we get that $\ell_{1}\in E(\overline{x})$. But as $E(\overline{x})\cap L= E$, then $\ell_{1}\in E$. Reversing the roles of $\ell_{1}$ and $\ell_{2}$, we also get that if $\ell_{2}\notin E$, then $\ell_{1}\notin E$. 
 
 If $\ell_{1},\ell_{2}\in E$, then, as $E$ is active, there is $h\in G^{E}$ such that $h\ell_{2}=\ell_{1}$, so we get that $\dim^{g}_{E}(\ell_{1},\ell_{2})=1$, proving (a). On the other hand, it is a simple exercise to see that, because $E$ is active, every point in $E$ is $E$-special, proving (b). 
 
So from now on we assume that $\ell_{1},\ell_{2}\in L\setminus E$. Note that we can take the entries of $g$ to be in $E[\overline{x}]$. Write:
    \begin{equation*}
        g=\left(\begin{matrix}
a(\overline{x}) & b(\overline{x})\\  
c(\overline{x}) & d(\overline{x})
\end{matrix}\right),
    \end{equation*}
    where $a(\overline{x}),b(\overline{x}),c(\overline{x}),d(\overline{x})\in E[\overline{x}]$, they are not all constant polynomials, and there is no common non-constant factor to all four polynomials. To set notation, write:
    \begin{equation*}
        \begin{matrix}
a(\overline{x}) =  \sum_{i}a_{i}\overline{x}^{i}, & b(\overline{x}) = \sum_{i}b_{i}\overline{x}^{i}, &  
c(\overline{x}) = \sum_{i}c_{i}\overline{x}^{i}, & d(\overline{x}) = \sum_{i}d_{i}\overline{x}^{i};
\end{matrix}
\end{equation*}
here we are using multi-index notation, so that $i$ is a tuple of non-negative integers. From the equation $\ell_{1} = g\ell_{2}$, we obtain:
\begin{equation*}
    \sum_{i}(c_{i}\ell_{1}\ell_{2}+d_{i}\ell_{1})\overline{x}^{i} = \sum_{i}(a_{i}\ell_{2}+b_{i})\overline{x}^{i}.
\end{equation*}
Given that $a_{i},b_{i},c_{i},d_{i}\in E\subseteq L$ and that $\overline{x}$ is algebraically independent over $L$, we deduce that the coefficient of $\overline{x}^{i}$ is zero for every $i$. So we obtain that for all $i$:
\begin{equation*}
    \ell_{1}(c_{i}\ell_{2}+d_{i}) = a_{i}\ell_{2}+b_{i}.
\end{equation*}
Let $g_{i}=\left(\begin{matrix}
a_{i} & b_{i}\\  
c_{i} & d_{i}
\end{matrix}\right)$. For some $i$ we must have $g_{i}\neq 0$, say this happens for $i=n$. If $\det(g_{n})=0$, this would mean that the rows of $g_{n}$ are linearly dependent over $E$. Suppose that $(c_{n},d_{n})\neq(0,0)$ (a similar argument can be made if instead we assume $(a_{n},b_{n})\neq(0,0)$). If $(c_{n}\ell_{2}+d_{n})=0$, then $\ell_{2}\in E$, which is a contradiction. So let us assume $(c_{n}\ell_{2}+d_{n})\neq 0$. As $\det(g_{n})= 0$, we may assume that the rows of $g_{n}$ are linearly dependent over $E$, i.e. there exists $\lambda_{n}\in E$ such that $a_{n}=\lambda_{n}c_{n}$ and $b_{n} = \lambda_{n}d_{n}$. But this would mean that $\ell_{1}\in E$, again arriving at a contradiction. 

So for all $i$ either $g_{i}=0$ or $\det(g_{i})\neq 0$. If every matrix $g_{i}$ is scalar, then so would be $g$. Therefore there is some $i$ for which $g_{i}$ is non-scalar, $\det(g_{i})\neq 0$ (so $g_{i}\in G^{E}$), and $g_{i}\ell_{2}  =\ell_{1}$. So $\ell_{1}$ and $\ell_{2}$ are $\mathrm{gcl}_{E}$-dependent, proving (a). Finally, if $\ell_{1}=\ell_{2}$, then $\ell_{1}$ is special over $E$, proving (b).
\end{proof}

\begin{lem}
\label{lem:disjoint2}
Let $E,F,L$ be subfields of $K$ such that $E\subseteq F\cap L$. Suppose $F\bot^{g}_{E}L$. Then for any tuple $\overline{x}$ from $K$ and any $A\subseteq L$, we have:
\begin{equation*}
    \dim^{g}_{F}(\overline{x}/L) - \dim^{g}_{E}(\overline{x}/L) \leq \dim^{g}_{F}(\overline{x}/A) - \dim^{g}_{E}(\overline{x}/A).
\end{equation*}
\end{lem}
\begin{proof}
Let $\overline{\ell}\in L$ be a tuple such that $\dim^{g}_{F}(\overline{x}/\overline{\ell}A) = \dim^{g}_{F}(\overline{x}/L)$ and $\dim^{g}_{E}(\overline{x}/\overline{\ell}A) = \dim^{g}_{E}(\overline{x}/L)$. The addition formula gives us two ways to calculate $\dim_{F}^{g}(\overline{x},\overline{\ell}/A)$:
\begin{eqnarray*}
\dim^{g}_{F}(\overline{x},\overline{\ell}/A) &=& \dim^{g}_{F}(\overline{x}/A) + \dim^{g}_{F}(\overline{\ell}/\overline{x}A)\\
&=& \dim^{g}_{F}(\overline{\ell}/A)+\dim^{g}_{F}(\overline{x}/\overline{\ell}A).
\end{eqnarray*}
Therefore:
\begin{eqnarray*}
\dim^{g}_{F}(\overline{x}/A) - \dim^{g}_{F}(\overline{x}/\overline{\ell}A) &=& \dim^{g}_{F}(\overline{\ell}/A) - \dim^{g}_{F}(\overline{\ell}/\overline{x}A)\\
&=& \dim^{g}_{E}(\overline{\ell}/A) - \dim^{g}_{F}(\overline{\ell}/\overline{x}A)\\
&\geq& \dim^{g}_{E}(\overline{\ell}/A) - \dim^{g}_{E}(\overline{\ell}/\overline{x}A)\\
&=& \dim^{g}_{E}(\overline{x}/A) - \dim^{g}_{E}(\overline{x}/\overline{\ell}A).
\end{eqnarray*}
\end{proof}

\section{$j$-Derivations}
\label{sec:jder}
In this section we introduce $j$-derivations and the pregeometry $j\mathrm{cl}$ they define. The exposition is analogous to section $4$ of \cite{kirby}.
\begin{defi}
Let $K$ be a field and $M$ a $K$-vector space. A map $\partial:K\rightarrow M$ is a called a \emph{derivation} if it satisfies for every $a,b\in K$:
\begin{enumerate}
\item $\partial(a+b) = \partial(a)+\partial(b)$.
\item $\partial(ab) = a\partial(b)+b\partial(a)$.
\suspend{enumerate}
Let $(K,D)$ be a $j$-field. The map $\partial:K\rightarrow M$ is called a \emph{$j$-derivation} if it is a derivation and it satisfies for every $z\in D$:
\resume{enumerate}
\item $\partial(j(z)) = j'(z)\partial(z)$,\quad $\partial(j'(z)) = j''(z)\partial(z)$,\quad $\partial(j''(z)) = j'''(z)\partial(z)$.
\end{enumerate}

For $C\subseteq K$, let $\mathrm{Der}(K/C, M)$ denote the set of derivations $\partial:K\rightarrow M$ such that $\partial(c) = 0$ for every $c\in C$. Let $j\mathrm{Der}(K/C, M)$ be the set of $j$-derivations $\partial:K\rightarrow M$ satisfying $\partial(c) = 0$ for every $c\in C$. For convenience we write $\mathrm{Der}(K/C):=\mathrm{Der}(K/C,K)$ and $j\mathrm{Der}(K/C) := j\mathrm{Der}(K/C,K)$. Note that all these spaces are $K$-vector spaces. 

Let $C$ be a subset of $K$. Define $\Omega(K/C)$ as the $K$-vector space generated by formal symbols of the form $dr$, where $r\in K$, quotiented by the relations given by the axioms of derivations plus that for every $c\in C$, $dc=0$. Denote by $d:K\rightarrow\Omega(K/C)$ the map $r\mapsto dr$. The map $d$ is called the \emph{universal derivation on $C$}. Let $\Xi(K/C)$ be obtained from $\Omega(K/C)$ by taking the quotient with the axioms for $j$-derivations. This induces a map $d_{j}:K\rightarrow \Xi(K/C)$ which is called the \emph{universal $j$-derivation}.
\end{defi}

\begin{prop}
Let $C\subseteq K$. For every $j$-derivation $\partial:K\rightarrow M$ which vanishes on $C$, there exists a unique $K$-linear map $\partial^{\ast}:\Xi(K/C)\rightarrow M$ such that $\partial^{\ast}\circ d_{j} = \partial$.
\end{prop}
\begin{proof}
For every $r\in K$, set $\partial^{\ast}(d_{j}r) = \partial(r)$ and then extend linearly. 
\end{proof}

Clearly, a similar proposition exists characterising the universal property of $\Omega(K/C)$. These universal properties in fact show that $\mathrm{Der}(K/C)$ and $j\mathrm{Der}(K/C)$ are the dual spaces of $\Omega(K/C)$ and $\Xi(K/C)$ respectively.
\begin{defi}
Let $C\subseteq K$, and $a\in K$. We say that $a$ belongs to the $j$-closure of $C$, denoted $a\in j\mathrm{cl}(C)$, if for every $\partial\in j\mathrm{Der}(K/C)$ we have that $\partial(a)=0$. That is:
\begin{equation*}
    j\mathrm{cl}(C) = \bigcap_{\partial\in j\mathrm{Der}(K/C)}\ker\partial.
\end{equation*}
\end{defi}
In particular this means that for any $A\subseteq K$ and any $\partial\in j\mathrm{Der}(K/A)$ we have that $\ker\partial$ is $j\mathrm{cl}$-closed. For any $C\subseteq K$, $F = j\mathrm{cl}(C)$ defines a $j$-subfield of $K$. Indeed, as every $j$-derivation is an additive homomorphism, $j\mathrm{cl}(C)$ is a group under addition. The second axiom of $j$-derivations show that it is closed under multiplication. Finally, as $1$ is in the kernel of any $j$-derivation, then the multiplicative inverse of every nonzero element of $F$ is in $F$. So $F$ is a field. The other axioms of $j$-derivations show that we can restrict $j$ to $F$ and still obtain a $j$-field. For this, first define $D_{F} = D\cap F$. Thus we can define $\alpha_{F}:D_{F}\rightarrow F$ as the restriction of $\alpha$ to $D_{F}$. Let $z\in D_{F}$. Let $\partial\in j\mathrm{Der}(K/C)$, then as $\partial(j(z)) = j'(z)\partial(\alpha(z))$ we get that $\partial(j(z)) = 0$, and so $j(z)\in F$. Similarly $j'(z), j''(z)\in F$. The fact that $j'''(z)\in F$ comes from axiom \ref{ax:diffj} of $j$-fields, and the fact that $F$ is relatively algebraically closed in $K$. The final thing to check is that $D_{F}$ is a $G$-set. But this follows because for every $g\in G$, $\alpha_{F}(gz)$ is algebraic over $F$ (because $\alpha_{F}(z)\in F$), and so $\partial(\alpha_{F}(gz))=0$. Therefore $(F,D_{F})$ is a $j$-subfield of $K$.

\begin{remark}
\label{rem:evilone}
Note that $a\in j\mathrm{cl}(C)$ if and only if $d_{j}a=0$ in $\Xi(K/C)$. Clearly if $d_{j}a=0$, then $a\in j\mathrm{cl}(C)$. Conversely suppose that $d_{j}a\neq 0$. Assuming $\Xi(K/C)\neq0$, we can choose a linear transformation $\eta:\Xi(K/C)\rightarrow K$ such that $\eta(d_{j}a)\neq 0$. It is a straightforward exercise to show that $\eta\circ d_{j}\in j\mathrm{Der}(K/C)$. Therefore, there exists $\partial\in j\mathrm{Der}(K/C)$ such that $\partial(a)\neq 0$. From this we can conclude that $j\mathrm{Der}(K/C) = j\mathrm{Der}(K/j\mathrm{cl}(C))$. 
\end{remark}

Recall the following standard results about derivations. 

\begin{lem}[see Chapter VIII, section 5 of \cite{lang}]
Let $K$ be a field and $\partial:K\rightarrow K$ be a derivation.
\begin{enumerate}
    \item Let $C\subseteq K$ and let $a\in K$ be algebraic over $\mathbb{Q}(C)$. Then every derivation which vanishes on $C$, also vanishes on $a$.
    \item Let $X$ be transcendental over $K$ and let $a\in K(X)$. Then $\partial$ can be extended to a derivation $\partial':K(X)\rightarrow K(X)$ such that $\partial'(X)=a$. 
    \item If $L$ is a separable algebraic extension of $K$, then $\partial$ extends to a derivation on $L$. 
\end{enumerate}
\end{lem}

\begin{lem}[see Lemma $6.7$ of \cite{marker-messmer-pillay}]
Let $K_{1}\subseteq K_{2}$ be a field extension. Then $\dim_{K_{2}}\Omega(K_{2}/K_{1}) = \mathrm{t.d.}(K_{2}/K_{1})$. 
\end{lem}

\begin{lem}[see Theorem $3$ of \cite{ax}]
\label{lem:linearlem}
Let $K$ be a field and $\Delta$ a set of derivations on $K$. Let $x_{1},\ldots,x_{n}\in K$ and set $r = \mathrm{rk}(\partial x_{i})_{\partial\in\Delta,i=1,\ldots,n}$. Then there is a set of derivations $\widetilde{\Delta}$ and $\partial'_{1},\ldots,\partial'_{n}\in\widetilde{\Delta}$ such that $\partial'_{i}(x_{k}) = \delta_{ik}$, the Kronecker delta, for every $\partial\in\widetilde{\Delta}\setminus\left\{\partial'_{1},\ldots,\partial'_{n}\right\}$ we have that $\partial(x_{i})=0$ for $i=1,\ldots,n$, and the elements of $\widetilde{\Delta}$ are $K$-linear combinations of the elements of $\Delta$. Furthermore $\bigcap_{\partial\in\Delta}\ker\partial = \bigcap_{\partial\in\widetilde{\Delta}}\ker\partial$. In particular $\mathrm{rk}(\partial x_{i})_{\partial\in\widetilde{\Delta},i=1\ldots,n} = r$. 
\end{lem}
\begin{proof}
Assume that $\partial_{1},\ldots,\partial_{r}\in\Delta$ and $x_{1},\ldots,x_{r}$ are such that: $\det(\partial_{i} x_{k})_{i,k=1,\ldots,r} \neq 0$. Let $A = (a_{ik}) = (\partial_{i}x_{k})_{i,k=1,\ldots,r}^{-1}$ and set $\partial'_{i} = \sum_{k=1}^{r}a_{ik}\partial_{k}\in\mathrm{Der}(F/C)$. In this way $\partial_{i}'(x_{k}) = \delta_{ik}$, the Kronecker delta. 
  
  For each $\partial\in\Delta$ there exist unique $b_{i}(\partial)\in F$ such that: $\partial(x_{k}) = \sum_{i=1}^{r}b_{i}(\partial)\partial_{i}(x_{k})$ for $k=1,\ldots,r$. Set $\widetilde{\partial} = \partial - \sum_{i=1}^{r}b_{i}(\partial)\partial_{i}$ and $\widetilde{\Delta} = \left\{\widetilde{\partial} : \partial\in\Delta\right\}\cup\left\{\partial'_{1},\ldots,\partial'_{r}\right\}$.
  
  Let $C = \bigcap_{\partial\in\Delta}\ker\partial$. Let $x\in K$ be such that $\partial(x)=0$ for every $\partial\in\widetilde{\Delta}$. In particular $\partial_{i}'(x)=0$, for $i=1,\ldots,r$. Given that $A$ is invertible, we get that $\partial_{i}(x)=0$ for $i=1,\ldots,r$. Therefore, for each $\partial\in\Delta$ we have that $0 = \widetilde{\partial}(x) = \partial(x) - \sum_{i=1}^{r}b_{i}(\partial)\partial_{i}(x) = \partial(x)$, which means that $x\in C$.  
  
  Conversely, it is immediate from definition that if $x\in C$, then $\partial(x)=0$ for every $\partial\in\widetilde{\Delta}$.
\end{proof}

\begin{defi}
Let $(K,D)$ be a $j$-field and let $\tau_{1},\ldots,\tau_{m}\in K$ be $j\mathrm{cl}$-independent. A \emph{corresponding system of $j$-derivations for $\overline{\tau}$} is a set of $j$-derivations $\partial_{1},\ldots,\partial_{m}\in j\mathrm{Der}(K)$ such that $\partial_{i}(\tau_{k}) = \delta_{ik}$, the Kronecker delta.
\end{defi}
So, what Lemma \ref{lem:linearlem} shows (if one restricts the proof so that it only uses $j$-derivations) is that for any $j\mathrm{cl}$-independent set, there is a corresponding system of $j$-derivations. Now we will focus on properties of $j$-derivations. The following Lemma is straightforward.

\begin{lem}[see Theorem 1.10 of \cite{wilkie}]
\label{lem:basic}
Given subsets $B,C\subseteq K$ we have that:
\begin{enumerate}[(a)]
\item $C\subseteq j\mathrm{cl}(C)$.
\item $B\subseteq C\implies j\mathrm{cl}(B)\subseteq j\mathrm{cl}(C)$.
\item $j\mathrm{cl}(j\mathrm{cl}(C)) = j\mathrm{cl}(C)$. 
\end{enumerate}
\end{lem}

\begin{lem}[Finite character of $j\mathrm{cl}$]
\label{lem:finchar}
Let $(K,D)$ be a $j$-field, $C\subseteq K$. If $a\in j\mathrm{cl}(C)$, then there is a finite set $C_{0}\subseteq C$ such that $a\in j\mathrm{cl}(C_{0})$. Furthermore, there is a finitely generated j-subfield $K_{0}$ of $K$ such that $C_{0}\subseteq K_{0}$ and $a\in j\mathrm{cl}_{K_{0}}(C_{0})$. 
\end{lem}
\begin{proof}
Let $\mathcal{L}$ be the language that consists purely of constant symbols, one for each element of $\Xi(K/\emptyset)$. Let $T$ be the $\mathcal{L}$-theory that says that these constant symbols satisfy the axioms of $K$-vector spaces, the axioms that say that $d_{j}$ is a $j$-derivations, and also $d_{j}c =0$ for each $c\in C$. Then $\Xi(K/C)\models T$ and $\Xi(K/C)\models d_{j}a = 0$ by Remark \ref{rem:evilone}. By the universal property of $\Xi(K/C)$, every other $j$-derivation $\partial:K\rightarrow M$ which vanishes over $C$ also satisfies $\partial(a)=0$. Also, every such $M$ is a model of $T$. Therefore $T\models d_{j}a=0$. By the Completeness Theorem $T\vdash d_{j}a=0$. There are only finitely many elements of $C$ used in the proof of $d_{j}a=0$. Let $T_{0}$ be the set of formulas used in the proof of $d_{j}a=0$ (thus, $T_{0}$ is finite). Let $C_{0}$ consist of the elements $c\in C$ such that the formula $d_{j}c=0\in T_{0}$. Let $K_{0}$ be the $j$-subfield of $K$ generated by all the $x\in K$ which appear in some axiom of $T_{0}$. Note that $C_{0}\subseteq K_{0}$. Then $d_{j}a=0$ in $\Xi(K_{0}/C_{0})$ and in $\Xi(K/C_{0})$. Therefore $a\in j\mathrm{cl}_{K_{0}}(C_{0})$ and $a\in j\mathrm{cl}_{K}(C_{0})$. 
\end{proof}

\begin{lem}[Exchange Principle for $j\mathrm{cl}$]
\label{lem:exchange}
Let $C\subseteq K$, $a,b\in K$. If $a\in j\mathrm{cl}(Cb)$ and $a\notin j\mathrm{cl}(C)$, then $b\in j\mathrm{cl}(Ca)$. 
\end{lem}
\begin{proof}
Suppose $a\in j\mathrm{cl}(Cb)$ and $b\notin j\mathrm{cl}(Ca)$. Then there is $\partial\in j\mathrm{Der}(K/C)$ such that $\partial(a)=0$ and $\partial(b)=1$. Let $\partial'\in j\mathrm{Der}(K/C)$ and let $\partial'' = \partial' - (\partial'(b))\partial$. Then $\partial''(a) = \partial'(a) - (\partial'(b))\partial(a) = \partial'(a)$. Also $\partial''(b) = \partial'(b) - (\partial'(b))\partial(b)=0$. So $\partial''\in j\mathrm{Der}(K/Cb)$, which means that $\partial''(a)=0$, therefore $\partial'(a)=0$. Therefore $a\in j\mathrm{cl}(C)$. 
\end{proof}

Combining these lemmas we get that $j\mathrm{cl}$ defines a pregeometry on any $j$-field. 
\begin{defi}
$C\subseteq K$ is said to be $j\mathrm{cl}$-\emph{closed} if $C = j\mathrm{cl}(C)$. $C$ is said to be \emph{$j$-independent} if it is an independent set with respect to $j\mathrm{cl}$. In particular, if $C$ is $j\mathrm{cl}$-independent and consists of only one element $C=\left\{\tau\right\}$, then we will say $\tau$ is \emph{$j$-generic}. Furthermore, for $A,C\subseteq K$, let $\dim^{j}(A/C)$ denote the dimension of $A$ over $C$ with respect to $j\mathrm{cl}$. 
\end{defi}
The next proposition should enlighten the reader as to the name chosen for the operator $j\mathrm{cl}$. 

\begin{prop}
\label{prop:jcl}
Let $C\subseteq K$ be $j\mathrm{cl}$-closed, and let $z\in D$.
\begin{enumerate}[(a)]
    \item $z\in C\implies\left\{j(z),j'(z),j''(z)\right\}\subseteq C$
    \item Suppose $j'''(z)\neq 0$. If $j^{(t)}(z)\in C$ for some $t\in\left\{0,1,2\right\}$, then $z\in C$.  
\end{enumerate}
\end{prop}
\begin{proof}
\begin{enumerate}[(a)]
    \item Let $\partial\in j\mathrm{Der}(K/C)$. If $\partial(z) = 0$, then  by axioms of $j$-derivations we get $\partial(j^{(t)}(z)) = j^{(t+1)}(z)\partial(z) = 0$, for $t\in\left\{0,1,2\right\}$. As $\partial$ was arbitrary in $j\mathrm{Der}(K/C)$, we are done.
    \item Let $\partial\in j\mathrm{Der}(K/C)$. Suppose first that $j''(z)\in C$. Then as $0 = \partial(j''(z))=j'''(z)\partial(z)$, we conclude $\partial(z)=0$. Assume now that $j'(z)\in C$. Then $0 = \partial(j'(z)) = j''(z)\partial(z)$. If $\partial(z)\neq 0$, then $j''(z)=0\in C$, but then we are in the previous case and we get a contradiction. So $\partial(z)=0$. Finally assume that $j(z)\in C$. Then $0=\partial(j(z)) = j'(z)\partial(z)$. Reasoning as before, if $\partial(z)\neq 0$, then $j'(z)=0\in C$, and so we are in the previous case, again arriving at a contradiction. Therefore, in all cases we arrive at $\partial(z)=0$.  As $\partial$ was arbitrary in $j\mathrm{Der}(K/C)$, we are done.
\end{enumerate}
\end{proof}

\begin{remark}
The special points over $\mathbb{Q}$ are in $j\mathrm{cl}(\emptyset)$. This is because the equation $gx=x$ is either a linear or quadratic equation over $\mathbb{Q}$, so special points are of degree $2$ over $\mathbb{Q}$. 
\end{remark}

\begin{prop}
\label{prop:dim}
Let $z_{1},\ldots,z_{n}\in D$ be such that $z_{1},\ldots,z_{n}$ are $j\mathrm{cl}$-independent over $C\subseteq K$. Then we have that $\dim^{g}(z_{1},\ldots,z_{n}/C)= n$. 
\end{prop}
\begin{proof}
Clearly $\dim^{g}(z_{1},\ldots,z_{n}/C)\leq n$. Suppose that $\dim^{g}(z_{1},\ldots,z_{n}/C)< n$. Then, without loss of generality, we may assume that $z_{1}\in \mathrm{gcl}(C,z_{2},\ldots,z_{n})$. This means that there is $x\in C\cup\left\{z_{2},\ldots,z_{n}\right\}$ and $g\in G$ such that $z_{1} = gx$. Every $j$-derivation which vanishes on $x$ vanishes also on $z_{1}$. This means that $z_{1}\in j\mathrm{cl}(x)$, which contradicts that $z_{1},\ldots,z_{n}$ are $j\mathrm{cl}$-independent over $C$. 
\end{proof}

\begin{lem}
Let $(K,D)$ be a $j$-field and $C\subseteq K$. Then $\dim^{j}(K/C) \leq \dim_{K}\Xi(K/C)$. 
\end{lem}
\begin{proof}
  Suppose $t_{1},\ldots,t_{n}$ are $j\mathrm{cl}$-independent over $C$. Then, as in Lemma \ref{lem:linearlem}, there are $j$-derivations $\partial_{i}:K\rightarrow K$ such that $\partial_{i}(t_{k}) = \delta_{ik}$, the Kronecker delta. Consider the universal $j$-derivation $d_{j}:K\rightarrow\Xi(K/C)$ and for each $i$ choose $\chi_{i}:\Xi(K/C)\rightarrow K$ such that $\partial_{i} = \chi_{i}\circ d_{j}$. Suppose that $a_{1},\ldots,a_{n}\in K$ satisfy:
  \begin{equation*}
      \sum_{i=1}^{n}a_{i}d_{j}t_{i} = 0.
  \end{equation*}
  Then, applying $\chi_{k}$ to this equation we get $\sum a_{i}\partial_{k}t_{i}=0$, that is to say $a_{k}=0$. Therefore $d_{j}t_{1},\ldots,d_{j}t_{n}$ are $K$-linearly independent. 
\end{proof}

To summarise the results about dimensions of universal derivations, we write:
\begin{equation*}
    \dim^{j}(K/C)\leq\dim_{K}\Xi(K/C) = \dim_{K}\Xi(K/j\mathrm{cl}(C))\leq\dim_{K}\Omega(K/j\mathrm{cl}(C))=\mathrm{t.d.}(K/j\mathrm{cl}(C)).
\end{equation*}

\subsection{$j$-derivations on $\mathbb{C}$}

In this section we prove that there are non-trivial $j$-derivations on $\mathbb{C}$. It is well known that $j$ has a standard fundamental domain: 
\begin{equation*}
    \mathcal{F} := \left\{z=x+iy\in\mathbb{H}^{+} : -\frac{1}{2}\leq x\leq 0, |z|\geq 1\right\}\cup\left\{z=x+iy\in\mathbb{H}^{+} : 0< x <\frac{1}{2}, |z|> 1\right\}. 
\end{equation*}
For every $z\in\mathbb{H}^{+}$ there exists $g\in\mathrm{SL}_{2}(\mathbb{Z})$ such that $gz\in\mathcal{F}$. Any set of the form $g\mathcal{F}$, with $g\in\mathrm{SL}_{2}(\mathbb{Z})$ is called a \emph{fundamental domain of $j$}. Now note that the differential equation (\ref{eq:diffj}) of $j$ is defined when $j'(z)\neq 0$ and $j(z)(j(z)-1728)\neq 0$. If we restrict $z$ to $\mathcal{F}$, then it is well known that:
\begin{itemize}
    \item if $j'(z)=0$, then $z=i$ or $z=\rho$, where $\rho = e^{2\pi i/3}$,
    \item if $j(z)=0$, then $z=\rho$,
    \item if $j(z)=1728$, then $z=i$. 
\end{itemize}

\begin{prop}
Let $\partial$ be a $j$-derivation on $\mathbb{C}$. Let $z\in\mathbb{H}^{+}\cup\mathbb{H}^{-}$ not be in the $\mathrm{SL}_{2}(\mathbb{Z})$-orbits of $i$, $-i$, $\rho$ and $\overline{\rho}$. Then for every $n\in\mathbb{N}$ we get that:
\begin{equation*}
    \partial(j^{(n)}(z)) = j^{(n+1)}(z)\partial(z).
\end{equation*}
\end{prop}
\begin{proof}
  To prove the result for $n\geq 3$, we express equation (\ref{eq:diffj}) in the following form:
  \begin{equation}
      \label{eq:j'''}
      j''' = \frac{3(j'')^{2}}{2j'} - \frac{j^{2}-1968j+2654208}{2j^{2}(j-1728)^{2}}(j')^{3}.
  \end{equation}
  Now we do two things to this equation: on one hand derive it (as holomorphic functions), on the other, apply $\partial$ to it. Compare both results to get that $\partial(j''') = j^{(4)}(z)\partial(z)$. By induction, we obtain the desired result.
\end{proof}

Given that $i$ and $\rho$ are algebraic over $\mathbb{Q}$, we get the following corollary.

\begin{cor}
\label{cor:zeros}
On $\mathbb{C}$, for every $n\in\mathbb{N}$, if $j^{(n)}(z)=0$, then $z\in j\mathrm{cl}(\emptyset)$. 
\end{cor}
\begin{proof}
  Suppose not, and let $\partial$ be a $j$-derivation such that $\partial(z)\neq 0$. The order of a zero of a holomorphic function is finite, so after applying $\partial$ to $j^{(n)}(z)$ enough consecutive times, we will get $\partial(z)=0$, a contradiction. 
\end{proof}

\begin{remark}
Recall that Schneider's theorem (see \cite{schneider}) says that, for $\tau\in\mathbb{H}$, both $\tau$ and $j(\tau)$ are algebraic over $\mathbb{Q}$ if and only if $\tau$ is special. This means then that if $\tau$ is algebraic over $\mathbb{Q}$ but not special, then $j(\tau)$ is transcendental over $\mathbb{Q}$. Then, in this case $j(\tau)$ is a transcendental number which is not $j$-generic. Indeed, if $\partial$ is a $j$-derivation on $\mathbb{C}$, then $\partial(j(\tau)) = j'(\tau)\partial(\tau)=0$.

Similarly, if we choose $\tau$ such that $j(\tau)$ is algebraic over $\mathbb{Q}$ but not the image of a special point, then necessarily $\tau$ must be transcendental over $\mathbb{Q}$. By Corollary \ref{cor:zeros} we get that $\tau$ is a transcendental number which is not $j$-generic. Both arguments give us families of complex numbers which are transcendental over $\mathbb{Q}$, but which are not $j$-generic. 

It is also known (see (2-c) of Theorem 9 of \cite{diaz}) that is $\tau\in\mathbb{H}$ is special and such that $j(\tau)\notin\left\{0,1728\right\}$, then $j'(\tau)$ is transcendental over $\mathbb{Q}$. So this gives yet another family. In the case of the exponential function there is an explicit equation relating $e$ and $\pi$ (i.e. $e^{\pi i}=-1$), which implies that $\pi$ is not exponentially transcendental. However, such a formula does not seem to be known in the case of the $j$-function. It is therefore not clear if we can decide whether or not $\pi i$ is $j$-generic or not. Results like Corollary 1.3 of \cite{nesterenko} may be seen as evidence that $j$, $j'$, $j''$ and $\pi$ are not related in a way that could allow us to prove that $\pi$ is not $j$-generic.
\end{remark}

Now we look for non-trivial $j$-derivations. The key idea here is that, to get $j$-derivations, we should look for derivations on $\mathbb{R}$ that respect the real and imaginary parts of $j$ on its fundamental domain. Let us set the following notation. If $F:U\rightarrow\mathbb{C}$ is a holomorphic function, where $U\subseteq\mathbb{C}$, then the real and imaginary parts of $F$, say $F_{1}$ and $F_{2}$ respectively, are the real analytic functions with domain:
\begin{equation*}
U_{\mathbb{R}} := \left\{(x,y)\in\mathbb{R}^{2} : x+iy\in U\right\}
\end{equation*}
satisfying $F(x+iy) = F_{1}(x,y) + iF_{2}(x,y)$, for $(x,y)\in U_{\mathbb{R}}$. 

Given two functions $f_{1},f_{2}:\mathbb{R}\rightarrow\mathbb{R}$, define the function $[f_{1}:f_{2}]:\mathbb{C}\rightarrow\mathbb{C}$ as $[f_{1}:f_{2}](x+iy) := (f_{1}(x)-f_{2}(y)) + i(f_{1}(y)+f_{2}(x))$ (for $x,y\in\mathbb{R}$). 
\begin{defi}
Let $\partial$ be a derivation on $\mathbb{C}$ and $F:U\rightarrow\mathbb{C}$ be a holomorphic function. We say that \emph{$\partial$ respects $F$ at a point $z\in U$} if: 
\begin{equation*}
    \partial(F(z)) = \frac{d F}{d z}(z)\partial(z).
\end{equation*}
If $F_{0}$ is the real or imaginary part of $F$ and $\partial'$ is a derivation on $\mathbb{R}$, then we say that \emph{$\partial'$ respects $F_{0}$ at $z=x+iy\in U$} if:
\begin{equation*}
    \partial'(F_{0}(x,y)) = \partial_{1}( F_{0}(x,y))\partial'(x) + \partial_{2}( F_{0}(x,y))\partial'(y),
\end{equation*}
where $\partial_{1},\partial_{2}$ are the partial derivatives of $F_{0}$ with respect to the first and second variables fo $F_{0}$ respectively. 
\end{defi}
The next lemma tells us that to obtain $j$-derivations on $\mathbb{C}$ we can use $j$-derivations on $\mathbb{R}$. 

\begin{lem}[see Lemma $4.2$ of \cite{wilkie}]
If $\partial_{1},\partial_{2}$ are derivations on $\mathbb{R}$, then $[\partial_{1}:\partial_{2}]$ is a derivation on $\mathbb{C}$. Further, if $F$ is a holomorphic function with real and imaginary parts $F_{1},F_{2}$ and domain $U\subseteq\mathbb{C}$, and if $\partial_{1},\partial_{2}$ respect $F_{1}$ and $F_{2}$ at a point $(x,y)\in U_{\mathbb{R}}$, then $[\partial_{1}:\partial_{1}]$ respects $F$ at the point $x+iy$.
\end{lem}

Let $F:U\rightarrow\mathbb{C}$ be a holomorphic function. The \emph{Schwarz reflection of $F$} is the holomorphic function $F^{SR}:U'\rightarrow\mathbb{C}$ given by $F^{SR}(z) := \overline{F(\overline{z})}$, where $U' = \left\{z : \overline{z}\in U\right\}$ and the bar denotes complex conjugation. The following theorem determines the kind of $j$-derivations we could hope to find on $\mathbb{C}$. Given a collection $\mathcal{C}$ of holomorphic functions, let $\mathrm{Der}_{\mathbb{C}}(\mathcal{C})$ denote the set of derivations on $\mathbb{C}$ which respect every function in $\mathcal{C}$, and let $\mathrm{Der}_{\mathbb{R}}(\mathcal{C}_{\mathrm{real}})$ be the set of derivations on $\mathbb{R}$ which respect the real and imaginary parts of the functions in $\mathcal{C}$. 

\begin{thm}[see Theorem $4.3$ of \cite{wilkie}]
Let $\mathcal{C}$ be any collection of holomorphic functions closed under Schwarz reflection and holomorphic derivation. Then the elements of $\mathrm{Der}_{\mathbb{C}}(\mathcal{C})$ are precisely the maps of the form $[\lambda:\mu]:\mathbb{C}\rightarrow\mathbb{C}$, for $\lambda,\mu\in\mathrm{Der}_{\mathbb{R}}(\mathcal{C}_{\mathrm{real}})$. 
\end{thm}

Let $\overline{\mathbb{R}}$ be the expansion of $\mathbb{R}$ with the real and imaginary parts of $j$, $j'$ and $j''$, with all these functions restricted to $\mathcal{F}$. It is noted in \cite{peterzil-starchenko} that the function $j:\mathcal{F}\rightarrow\mathbb{C}$ is definable in $\mathbb{R}_{\mathrm{an},\exp}$. So the structure $\overline{\mathbb{R}}$ is o-minimal (for definitions and basic results of o-minimality see e.g. \cite{macpherson}). The next lemma is a basic result from complex analysis.

\begin{lem}
\label{lem:complex}
Let $f:U\rightarrow\mathbb{C}$ be a holomorphic function. Write $f(x+iy) = u(x,y) + iv(x,y)$, for $(x,y)\in U_{\mathbb{R}}$. Then:
\begin{equation*}
    \mathcal{R}e\left(\frac{df}{dz}\right) = u_{x}, \qquad \mathcal{I}m\left(\frac{df}{dz}\right) = v_{x},
\end{equation*}
where $u_{x}$ and $v_{x}$ denote the partial derivatives of $u$ and $v$ with respect to the first variable.
\end{lem}

\begin{prop}
\label{prop:jderonr}
The structure $\overline{\mathbb{R}}$ has non-trivial derivations that respect the real and imaginary parts of $j$, $j'$ and $j''$ restricted to the interior of $\mathcal{F}$. 
\end{prop}
\begin{proof}
  Let $\mathrm{dcl}$ denote the definable closure operator of the o-minimal structure $\overline{\mathbb{R}}$ (see section 2.2 of \cite{macpherson} for definitions). Choose $\tau\in\mathbb{R}$ such that $\tau\notin\mathrm{dcl}(\emptyset)$. Let $B\subseteq\mathbb{R}$ be such that $\left\{\tau\right\}\cup B$ is a basis of $\overline{\mathbb{R}}$ with respect to $\mathrm{dcl}$, so in particular $\mathrm{dcl}(\tau, B)=\mathbb{R}$. Given $a\in\mathbb{R}$ there is a tuple $\overline{b}$ of $B$ and a formula $\phi_{a}(\tau,\overline{b},x)$ such that $a$ is the only point satisfying $\phi_{a}$. By o-minimality, the formula $\psi(y)\equiv\exists !x(\phi_{a}(y,\overline{b},x))$ partitions $\mathbb{R}$ into finitely many intervals. As $\tau$ satisfies $\psi(y)$, but $\tau\notin\mathrm{dcl}(B)$, then $\tau$ is not an end-point of any of these intervals. Then there exists a function $f_{a}(t):\mathbb{R}\rightarrow\mathbb{R}$ such that around $\tau$, $f_{a}(y)=x\iff\phi_{a}(y,\overline{b},x)$. In particular then $f_{a}(\tau)=a$. Using the $\epsilon-\delta$ definitions of continuity and derivation one can also prove that $f_{a}(t)$ is differentiable in an interval around $\tau$ (this is a standard o-minimal result, for reference see Theorem 4.2 of \cite{pillay-steinhorn}).  
  
  Define $\partial:\mathbb{R}\rightarrow\mathbb{R}$ as $\partial(a) = \frac{df_{a}}{dt}(\tau)$. Of course, one should ask if this is well defined. It may be that there exist many different functions satisfying the properties of $f_{a}$. But if $g(t)$ is another such function, then using o-minimality and looking at the set $\left\{t\in\mathbb{R} : f_{a}(t)=g(t)\right\}$, one concludes that there is an interval around $\tau$ on which both functions agree (because $\tau\notin\mathrm{dcl}(B)$). Therefore $\partial$ is well defined.
  
  Now we show that $\partial$ is a derivation which respects $j$, $j'$ and $j''$ on the interior of $\mathcal{F}$. Let $a_{1},a_{2}\in\mathbb{R}$. Note that if we have functions $f_{a_{1}}$ and $f_{a_{2}}$, then we can choose $f_{a_{1}+a_{2}}$ to be $f_{a_{1}}+f_{a_{2}}$. Similarly, we can choose $f_{a_{1}a_{2}}$ as $f_{a_{1}}f_{a_{2}}$. This proves that $\partial$ is a derivation. 
  
  Let $j_{1},j_{2}$ be the real and imaginary parts of $j$ respectively. Let $z$ be in the interior of $\mathcal{F}$, and set $x$ and $y$ as the real and imaginary parts of $z$. To calculate $\partial(j_{1}(z))$ we consider the function $g(t) = j_{1}(f_{x}(t),f_{y}(t))$, where we have changed our notation as we now consider $j_{1}$ as a function $j_{1}:\mathbb{R}^{2}\rightarrow\mathbb{R}$. These changes between $\mathbb{R}^{2}$ and $\mathbb{C}$ will be used many times, but they are very natural, so no confusion should arise.
  \begin{equation*}
      \frac{dg}{dt}(\tau) = \partial_{1}(j_{1})\frac{df_{x}}{dt}(\tau) + \partial_{2}(j_{1})\frac{df_{y}}{dt}(\tau),
  \end{equation*}
  where $\partial_{1}$ and $\partial_{2}$ denote the real analytic partial derivatives of $j_{1}$ with respect to each of its variables. By Lemma \ref{lem:complex} we get that $\partial_{1}(j_{1}) = j'_{1}$, where $j'_{1}$ is the real part of $j'$. By the Cauchy-Riemann equations we also get $\partial_{2}(j_{1}) = -\partial_{1}(j_{2})$, and again by Lemma \ref{lem:complex} we get $\partial_{1}(j_{2}) = j'_{2}$, where $j_{2}$ is the imaginary part of $j'$. Therefore:
  \begin{equation*}
      \partial(j_{1}(z)) = j'_{1}(z)\partial(x)-j'_{2}(z)\partial(y).
  \end{equation*}
  Note that this condition is precisely what you get when you look at the real part of the equation $\partial(j(z)) = j'(z)\partial(z)$. Similarly one shows that $\partial(j_{2}(z)) = j'_{2}(z)\partial(x) + j_{1}(z)\partial(y)$. An analogous argument shows that $\partial$ respects the real and imaginary parts of $j'$ and $j''$ restricted to $\mathcal{F}$. 
\end{proof}

\begin{thm}
There is a non-trivial $j$-derivation on $\mathbb{C}$.
\end{thm}
\begin{proof}
  Let $\partial$ be a non-trivial derivation on $\mathbb{R}$ given by Proposition \ref{prop:jderonr}. We now consider the derivation $[\partial:0]$ on $\mathbb{C}$. As this derivation extends $\partial$, we will keep the same name, so that $\partial=[\partial:0]$. We show that $\partial$ is in fact a $j$-derivation. It is clear that $\partial$ satisfies the axioms of $j$-derivation for every $z$ in the interior of the standard fundamental domain $\mathcal{F}$. Choose now any $z\in\mathbb{H}$. 
  
  Choose $g\in G$ such that $gz$ is in the interior of $\mathcal{F}$ and $(j(z),j(gz))$ is not a singular point of $\Phi_{N}(X,Y)=0$ (where $N = \det(N(g)g)$). When $z$ is not in the boundary of a fundamental domain, we can choose $g\in \mathrm{SL}_{2}(\mathbb{Z})$, and as $\Phi_{1}(X,Y) = X-Y$ we are guaranteed that $(j(z),j(gz))$ is not a singularity. If $z$ is in the boundary of a fundamental domain, then we can choose $g$ so that $N=2$, as $\Phi_{2}(X,Y)=0$ also does not have singularities in $\mathbb{C}^{2}$ (this can be seen as a consequence of Proposition 3.1 and Corollary 3.1 of \cite{klyachko}, or can be proven directly using the explicit form of $\Phi_{2}$ given, for example, in page 70 of \cite{zagier}). 
  
  Now that we have that, we consider the equation $\Phi_{N}(j(z),j(gz))=0$. Deriving with respect to $z$ gives us:
  \begin{equation}
  \label{eq:loc1}
  \Phi_{N1}(j(z),j(gz))j'(z) + \Phi_{N2}(j(z),j(gz))j'(gz)\frac{ad-bc}{(cz+d)^{2}} =0, 
  \end{equation}
  where $a,b,c,d$ are the entries of $g$ in the usual way, and $\Phi_{N1}$ and $\Phi_{N2}$ are the partial derivatives of $\Phi_{N}(X,Y)$ with respect to $X$ and $Y$ respectively. On the other hand, applying $\partial$ to the equation $\Phi_{N}(j(z),j(gz))=0$ (and using that $\partial$ respects $j$ and its derivatives inside of $\mathcal{F}$) gives us:
    \begin{equation}
    \label{eq:loc2}
  \Phi_{N1}(j(z),j(gz))\partial(j(z)) + \Phi_{N2}(j(z),j(gz))j'(gz)\frac{ad-bc}{(cz+d)^{2}}\partial(z) =0. 
  \end{equation}
  Observe now that if $\Phi_{N1}(j(z),j(gz))= 0$, then using the fact that $ad-bc=\det(g)\neq 0$, that $gz$ is in the interior of $\mathcal{F}$, and that the zeros of $j'$ are in the boundary of $\mathcal{F}$, we get from equation (\ref{eq:loc1}) that $\Phi_{N2}(j(z),j(gz))=0$, i.e. $(j(z),j(gz))$ is a singular point of $\Phi_{N}(X,Y)$, which contradicts our assumptions. So $\Phi_{N1}(j(z),j(gz))\neq 0$.
  
  Consider now equation (\ref{eq:loc2}). If $\partial(z)=0$, then, by the previous paragraph, $\partial(j(z))=0$. If $\partial(z)\neq0$, then comparing equations (\ref{eq:loc1}) and (\ref{eq:loc2}), we get that $j'(z) = \partial(j(z))/\partial(z)$. In both cases we obtain $\partial(j(z)) = j'(z)\partial(z)$.
  
  To verify the other axioms of $j$-derivations, we proceed analogously. We take equation (\ref{eq:loc1}), we differentiate it with respect to $z$ and we also apply $\partial$ to it. We compare both results. Using that  $\partial(j(z)) = j'(z)\partial(z)$ and that $(j(z),j(gz))$ is not a singularity of $\Phi_{N}$, one concludes that $\partial(j'(z))=j''(z)\partial(z)$. A similar argument will show that $\partial(j''(z))=j'''(z)\partial(z)$.

  Finally, if we have $z\in\mathbb{H}^{-}$, we have to work with $-z$ and proceed as above. So $\partial$ is a non-trivial $j$-derivation on $\mathbb{C}$. 
\end{proof}

\begin{remark}
\label{rem:countable}
As $i$ is algebraic over $\mathbb{Q}$, then for any derivation $\partial$ on $\mathbb{C}$ we have that $\partial(x+iy) = \partial(x)+i\partial(y)$. So, given $z=x+iy\in\mathbb{C}$ such that $x$ or $y$ is not in $\mathrm{dcl}(\emptyset)$ (where $\mathrm{dcl}$ is taken in $\overline{\mathbb{R}}$), there is a $j$-derivation which does not vanish on $z$. Therefore $j\mathrm{cl}(\emptyset)$ is contained in the set of complex numbers whose real and imaginary parts are both in $\mathrm{dcl}(\emptyset)$. Since this set is countable, we have that $j\mathrm{cl}(\emptyset)$ is countable. 
\end{remark}

A more general study of derivations that respect given holomorphic functions can be found in sections 2, 3 and 4 of \cite{jones-kirby-servi}.

\section{A Schanuel Property for $j$}

In this section we are interested in results relating the transcendence degree of $j$ and its derivatives with respect to the dimensions $\dim^{j}$ and $\dim_{F}^{g}$. The exposition is inspired by \cite{bays-kirby-wilkie}. First, let us introduce a ``modular version'' of Schanuel's conjecture. This is a special case of the \emph{generalised period conjecture} of Grothendieck-Andr\'e (see \cite{andre}, \cite{bertolin} and \cite{pila2}).

\begin{conj}[Modular Schanuel Conjecture: MSC]
Suppose $z_{1},\ldots,z_{n}\in \mathbb{H}^{+}$ are $\mathrm{gcl}$-independent and non-special. Then:
\begin{equation}
\label{eq:msc}
    \mathrm{t.d.}(\overline{z},j(\overline{z}),j'(\overline{z}),j''(\overline{z}))\geq 3n.
\end{equation}
\end{conj}

Following the notation we have been using so far, whenever we write $\mathrm{t.d.}(A)$, we mean the transcendence degree of the field $\mathbb{Q}(A)$ over $\mathbb{Q}$. The main ingredient for our transcendence results is the following.

\begin{thm}[Ax-Schanuel for $j$, see Theorem $1.3$ in \cite{pila-tsimerman}]
\label{thm:ax-schanuel-j}
Let $K$ be a characterstic zero differential field with $m$ commuting derivations $\partial_{k}$. Let $C = \bigcap_{k}\ker\partial_{k}$. Let $z_{i}, j_{i}, j'_{i}, j''_{i}, j'''_{i}\in K^{\ast}$, $i=1,\ldots,n$, be such that
\begin{equation*}
\partial_{k}j_{i} = j'_{i}\partial_{k}z_{i},\qquad \partial_{k}j'_{i} = j''_{i}\partial_{k}z_{i},\qquad \partial_{k}j''_{i}=j_{i}'''\partial_{k}z_{i}
\end{equation*}
for all $i$ and $k$. Suppose further that $\daleth(j_{i}, j'_{i}, j''_{i}, j'''_{i})=0$ for $i=1,\ldots,n$, where $\daleth$ is given by equation (\ref{eq:j}). Suppose also that for every modular polynomial $\Phi_{N}$ we have $\Phi_{N}(j_{i},j_{k})\neq 0$ for all $i$, $k$, $N$ and $j_{i}\notin C$ for all $i$. Then:
\begin{equation*}
\mathrm{t.d.}(z_{1},j_{1},j'_{1},j''_{1},\ldots,z_{n},j_{n},j'_{n},j''_{n}/C)\geq 3n + \mathrm{rank}(\partial_{k}z_{i})_{i,k}.
\end{equation*}
\end{thm}

This theorem is the first instance where we actually need the explicit shape of the differential equation (\ref{eq:diffj}) and that of the modular polynomials. Up until now, we have only used that they exist, some smoothness properties, and that they are defined over $\mathbb{Q}$. 

\subsection{Consequences of Ax-Schanuel for $j$}

\begin{prop}
\label{prop:as}
Let $(K,D)$ be a $j$-field. Let $z_{1},\ldots,z_{n}\in D$ and let $C\subseteq K$ be $j\mathrm{cl}$-closed. Then:
\begin{equation*}
\mathrm{t.d.}(\overline{z},j(\overline{z}),j'(\overline{z}),j''(\overline{z})/C) \geq 3\dim^{g}(\overline{z}/C) + \dim^{j}(\overline{z}/C).
\end{equation*}
\end{prop}
\begin{proof}
 The strategy of the proof is first to reduce the problem to a (possibly) different $j$-field, and then check conditions on the new $j$-field to see that we can use the Ax-Schanuel Theorem for $j$. Let $m=\dim^{j}(\overline{z}/C)$. By renumbering, we can assume that $z_{1},\ldots,z_{m}$ are $j\mathrm{cl}$-independent over $C$. Also, by renumbering and using the fact that $j\mathrm{cl}$-independence implies $\mathrm{gcl}$-independence, we can assume that $z_{1},\ldots,z_{r}$ are $\mathrm{gcl}$-independent over $C$, with $r = \dim^{g}(\overline{z}/C)$. Note that $m\leq r$. As in Lemma \ref{lem:linearlem}, we can find a corresponding system of $j$-derivations $\partial_{1},\ldots,\partial_{m}\in j\mathrm{Der}(K/C)$. Thus $m=\mathrm{rank}(\partial_{k}(z_{i}))_{i,k=1,\ldots,r}$.
 
 For $m<i\leq r$ we have that $z_{i}\in j\mathrm{cl}(z_{1},\ldots,z_{m},C)$, but $z_{i}\notin C$ (because $C$ is $j\mathrm{cl}$-closed). By the exchange property $z_{1}\in j\mathrm{cl}(z_{2},\ldots,z_{m},z_{i},C)$. As $\partial_{1}(z_{1}) =1$, then we must necessarily have $\partial_{1}(z_{i})\neq 0$. So in fact, $\partial_{k}(z_{i})\neq 0$ for every $k=1,\ldots,m$ and $i=m+1,\ldots, r$.  
 
 Let $F = j\mathrm{cl}(z_{1},\ldots,z_{n},C)$. We know that $F$ is a $j$-field. We want to show that the $j$-derivations $\partial_{1},\ldots,\partial_{m}$ can be restricted to $F$. Consider $\partial_{1}$. Clearly $\partial_{1}$ can be restricted to $C$, because $\partial_{1}(C)=0$ and $0\in C$. As a field, $C(z_{1},\ldots,z_{m})$ is isomorphic to the field of rational functions on $C$ with $m$ variables. Then $\partial_{1}$ acts on $C(z_{1},\ldots,z_{m})$ as partial derivation with respect to $z_{1}$, so it can also be restricted to $C(z_{1},\ldots,z_{m})$. Let $x\in F$ be algebraic over $C(z_{1},\ldots,z_{m})$, and let $p(x)=0$ be an algebraic equation with coefficients in $C(z_{1},\ldots,z_{m})$ satisfied by $x$. Applying $\partial_{1}$ to $p(x)=0$ shows that $\partial_{1}(x)\in C(z_{1},\ldots,z_{m},x)$. Finally, suppose that $z\in D_{F}$ is such that $\partial_{1}(z)\in F$. Then, as $F$ is a $j$-field, the axioms of $j$-derivations say that $\partial_{1}(j(z)),\partial_{1}(j'(z)),\partial_{1}(j''(z))\in F$. Therefore, $\partial_{1}$ can be restricted to a $j$-derivation on $F$. The most important thing about this is that, as $z_{i}\in F$, for $m<i\leq r$, then $\partial_{1}(z_{i})\in F$. This means that $z_{1},\ldots,z_{m}$ are still $j\mathrm{cl}$-independent over $C$ in $F$. Also, $z_{1},\ldots,z_{r}$ are clearly $\mathrm{gcl}$-independent over $C$ in $F$. 
 
 So, in what remains of the proof, we will work in the $j$-field $F$ while preserving the notation we have already set. We now show that in $F$ we have the conditions necessary for applying the Ax-Schanuel Theorem for $j$.
 
 Let $B = \bigcap_{k=1}^{m}\ker(\partial_{k})$ (remember that the kernels are taken inside $F$). Then $z_{i}\notin B$, for $i=1,\ldots,r$. Note that $C\subseteq B$ and that $z_{1},\ldots,z_{m}$ are still $j\mathrm{cl}$-independent over $B$. Given $i$ and $k$, consider the $j$-derivation $\partial = \partial_{i}\partial_{k} - \partial_{k}\partial_{i}$. Consider $\partial$ as a $j$-derivation on $K$. As $\partial$ vanishes on $C\cup\left\{z_{1},\ldots,z_{m}\right\}$, then for any $x\in K$ such that $x\in j\mathrm{cl}(z_{1},\ldots,z_{m},C)$, we must have that $\partial(x)=0$. In other words, $\partial\equiv 0$ on $F$. This means that the $j$-derivations $\partial_{1},\ldots,\partial_{m}$ commute on $F$.
 
 Let $j_{i} = j(z_{i})$, $j'_{i}=j'(z_{i})$, $j''_{i}=j''(z_{i})$ and $j'''_{i} = j'''(z_{i})$, for $i=1,\ldots,r$. Now we verify the conditions on Theorem \ref{thm:ax-schanuel-j}. Note that the system of differential equations $\partial_{k}(j_{i}) = j'_{i}\partial_{k}(z_{i})$, $\partial_{k}(j'_{i}) = j''_{i}\partial_{k}(z_{i})$, $\partial_{k}(j''_{i}) = j'''_{i}\partial_{k}(z_{i})$, for all $i=1,\ldots,r$, $k=1,\ldots,m$ is satisfied by the axioms of $j$-derivations. By the axioms of $j$-fields we have that $\daleth(j_{i}, j'_{i}, j''_{i}, j'''_{i})=0$, $i=1,\ldots,r$ (where $\daleth$ is given by equation (\ref{eq:j}). Note that $B$ contains all special points, so $z_{i}$ is not special and $\daleth$ is well defined at every $z_{i}$, $i=1,\ldots,r$. By $\mathrm{gcl}$-independence we have that $\Phi_{N}(j_{i},j_{k})\neq 0$ for all $1\leq i,k\leq r$,and all $N$.
 
 By Proposition \ref{prop:jcl}, if $j'''_{i}\neq 0$ for some $i$, then $j_{i}\notin B$ and $z_{i},j_{i},j'_{i},j''_{i},j'''_{i}\in F^{\ast}$ (because $0\in B$). So now we show $j'''_{i}\neq0$ for  $i=1,\ldots,r$. If $j'''_{k}=0$ for some $1\leq k\leq r$, then for every $j$-derivation $\partial$ we have that $\partial(j''_{k})=0$. So $j''_{k}\in B$. As $\partial(j'_{k}) = j''_{k}\partial(z_{k})$, then $j'_{k} = j''_{k}z_{k}+b_{0}$, for some $b_{0}\in B$. Therefore $\partial(j_{k}) = (j''_{k}z_{k}+b_{0})\partial(z_{k}) = \partial(j''_{k}z_{k}^{2}/2+b_{0}z_{k})$, which means there is $b_{1}\in B$ such that $j_{k} = j''_{k}z_{k}^{2}/2 + b_{0}z_{k}+b_{1}$. Plugging this values into the equation $\daleth(j_{k}, j'_{k}, j''_{k}, j'''_{k})=0$, we conclude that $z_{k}$ is algebraic over $B$, but as $B$ is algebraically closed, this would imply that $z_{k}\in B$, which is impossible for $1\leq k\leq r$.
 
 We can therefore apply Theorem \ref{thm:ax-schanuel-j} to obtain:
\begin{equation*}
\mathrm{t.d.}(z_{1},j_{1},j'_{1},j''_{1},\ldots,z_{r},j_{r},j'_{r},j''_{r}/B)\geq 3r + m.
\end{equation*}
Now use the general fact that if $A\subseteq A'$ and $C\subseteq B$, then $\mathrm{t.d.}(A'/C)\geq\mathrm{t.d.}(A/B)$. 
\end{proof}

The following lemma will allow us to obtain a weaker version of our main result, but the proof uses a different method. The argument presented here is based on ideas of Alex Wilkie for treating the exponential function, communicated to the author. 

\begin{lem}
\label{lem:weak}
Let $(K,D)$ be a $j$-field and $\partial$ be a $j$-derivation on $K$. Let $m\geq 1$, $d\geq 0$ and set $C=\ker\partial$. Suppose that:
\begin{enumerate}
    \item $z_{1}\ldots,z_{m}\in D$ are $\mathrm{gcl}$-independent, 
    \item There is $\tau\in K$ such that $\partial(\tau)=1$,
    \item $d = \dim^{g}_{F}(z_{1},\ldots,z_{m})$, where $F=\mathbb{Q}(\tau)$. 
\end{enumerate}
Then:
\begin{equation*}
    \mathrm{t.d.}(j(\overline{z}),j'(\overline{z}),j''(\overline{z}))\geq 3m-3d.
\end{equation*}
\end{lem}
\begin{proof}
  Given that if $d\geq m$ then the result is trivial, we assume that $d<m$. Renumbering if necessary, let $0\leq p\leq m$ be maximal so that $z_{1},\ldots,z_{p}\in C$. Note that by Lemma \ref{lem:disjoint1} $F\bot^{g}_{\mathbb{Q}}C$, so $z_{1},\ldots,z_{p}$ are $\mathrm{gcl}_{F}$-independent. Therefore $p\leq d$. 
  
  Now, by renumbering again, we can assume that $z_{1},\ldots,z_{p},z_{p+1},\ldots,z_{d}$ is an $\mathrm{gcl}_{F}$-basis for $z_{1},\ldots,z_{m}$. Then:
  \begin{eqnarray*}
  \mathrm{t.d.}(z_{p+1},\ldots,z_{m}/C)&\leq&\mathrm{t.d.}(z_{p+1},\ldots,z_{m},\tau/C)\\
  &=& \mathrm{t.d.}(\tau/C) + \mathrm{t.d.}(z_{p+1},\ldots,z_{m}/C,\tau)\\
  &\leq& 1 + d-p.
  \end{eqnarray*}
  
  Let $\widetilde{z} = (z_{p+1},\ldots,z_{m})$. Suppose that the Lemma is false: $\mathrm{t.d.}(j(\widetilde{z}),j'(\widetilde{z}),j''(\widetilde{z})/C)\leq 3m-3d-1$. Then it is also true that: $\mathrm{t.d.}(j(\widetilde{z}),j'(\widetilde{z}),j''(\widetilde{z})/C,\widetilde{z})\leq 3m-3d-1$.
  Therefore:
  \begin{eqnarray*}
      \mathrm{t.d.}(\widetilde{z},j(\widetilde{z}),j'(\widetilde{z}),j''(\widetilde{z})/C) &=& \mathrm{t.d.}(\widetilde{z}/C) + \mathrm{t.d.}(j(\widetilde{z}),j'(\widetilde{z}),j''(\widetilde{z})/C,\widetilde{z})\\
      &\leq& (1+d-p) + (3m-3d-1) = 3m-p-2d.
  \end{eqnarray*}
  However, under the hypothesis of the Lemma, we should have by the Ax-Schanuel theorem for $j$ (or by Proposition \ref{prop:as}) that:
  \begin{equation*}
      \mathrm{t.d.}(\widetilde{z},j(\widetilde{z}),j'(\widetilde{z}),j''(\widetilde{z})/C)\geq 3m-3p+1,
  \end{equation*}
  which would mean that $3m-3p+1 \leq 3m-p-2d$, or equivalently $2d+1\leq 2p$. This contradicts that $p\leq d$ (which was already established). 
\end{proof}

For the following Corollary recall Remark \ref{rem:countable}, which ensures the existence of $\tau$. 

\begin{cor}[Weak version of the main result]
\label{cor:weak}
Choose $\tau\in\mathbb{R}$ so that $\tau\notin j\mathrm{cl}(\emptyset)$. Suppose $z_{1},\ldots,z_{n}\in\mathbb{H}$ and $g\in G(\mathbb{Q}(\tau))$ are such that $z_{1},\ldots,z_{n},gz_{1},\ldots,gz_{n}$ are $\mathrm{gcl}$-independent. Then:
\begin{equation*}
    \mathrm{t.d.}(j(\overline{z}),j(g\overline{z}),j'(\overline{z}),j'(g\overline{z}),j''(\overline{z}),j''(g\overline{z}))\geq 3n.
\end{equation*}
\end{cor}
\begin{proof}
  We use Lemma \ref{lem:weak} with $m=2n$. Proposition \ref{prop:jderonr} constructs a $j$-derivation $\partial$ satisfying $\partial(\tau)=1$. Note also that in this case $d \leq n$. So $3m-3d \geq 6n - 3n = 3n$. 
\end{proof}

\subsection{Transcendence Results}

Proposition \ref{prop:as} can be rephrased as follows (which, even though a more convoluted statement, will be better for obtaining the main theorem).

\begin{cor}
\label{cor:reform}
Let $(K,D)$ be a $j$-field and $C\subseteq K$ be $j\mathrm{cl}$-closed. Let $\tau_{1},\ldots,\tau_{m}\in D$ be such that they are  $j\mathrm{cl}$-independent over $C$. Then for any $z_{1},\ldots,z_{n}\in D$:
\begin{equation*}
\mathrm{t.d.}(\overline{\tau},\overline{z},j(\overline{\tau}), j(\overline{z}), j'(\overline{\tau}), j'(\overline{z}), j''(\overline{\tau}), j''(\overline{z})/C) \geq 3\dim^{g}(\overline{\tau},\overline{z}/C) + m.
\end{equation*}
\end{cor}

Next, we need a couple of technical results.

\begin{prop}
\label{prop:moreresults}
Let $(K,D)$ be a $j$-field, $C\subseteq K$ $j\mathrm{cl}$-closed. Suppose $\tau_{1},\ldots,\tau_{m}\in D$ are  $j\mathrm{cl}$-independent over $C$. Then for any $z_{1},\ldots,z_{n}\in D$:
\begin{enumerate}[(a)]
\item \begin{equation*}
    \mathrm{t.d.}(j(\overline{\tau}), j'(\overline{\tau}), j''(\overline{\tau}) /C,\overline{\tau},\overline{z}, j(\overline{z}), j'(\overline{z}),j''(\overline{z}))\leq 3\dim^{g}(\overline{\tau}/C,\overline{z}).
\end{equation*}
\item \begin{equation*}
    \mathrm{t.d.}(\overline{z}/C,\overline{\tau})\leq\dim_{\mathbb{Q}(\overline{\tau})}^{g}(\overline{z}/C).
\end{equation*}
\end{enumerate}
\end{prop}
\begin{proof}
\begin{enumerate}[(a)]
    \item Suppose that $\tau_{1}\in \mathrm{gcl}(\tau_{2},\ldots,\tau_{m},\overline{z},C)$. Then, by $j\mathrm{cl}$-independence, we have that $\tau_{1} = gz_{i}$, for some $z_{i}$ and some $g\in G$. This says that $\tau_{1}\in\mathbb{Q}(z_{i})$. Also, for some $N$, we have that $\Phi_{N}(j(\tau_{1}),j(z_{i}))=0$, which says that $j(\tau_{1})$ is algebraic over $\mathbb{Q}(j(z_{i}))$. 
    
    Now we consider $j$-derivations. Given that $C$ contains the field of constants and that $z_{i}\notin C$ (because  $\tau_{1} = gz_{i}$), we can assume that there is a $j$-derivation $\partial$ such that $\partial(z_{i})\neq 0$. So we can assume that $\partial(z_{i})=1$. From the equality: 
    \begin{equation}
    \label{eq:partial1}
    \partial(\tau_{1}) = \frac{ad-bc}{(cz_{i}+d)^{2}}\partial(z_{i})
    \end{equation}
     (where $a$, $b$, $c$ and $d$ are the parameters of $g$ in the usual way), we get that $\partial(\tau_{1})$ is algebraic over $\mathbb{Q}(z_{i})$. Also:
    \begin{equation}
    \label{eq:partial2}
         0= \partial\Phi_{N}(j(\tau_{1}),j(z_{i})) = \Phi_{N,1}(j(\tau_{1}),j(z_{i}))j'(\tau_{1})\partial(\tau_{1}) + \Phi_{N,2}(j(\tau_{1}),j(z_{i}))j'(z_{i})\partial(z_{i}),
    \end{equation}
    where $\Phi_{N,1}$ and $\Phi_{N,2}$ are the derivatives of $\Phi_{N}$ with respect to the first and second coordinates respectively. We deduce that $j'(\tau_{1})$ is algebraic over $\mathbb{Q}(j(\tau_{1}), j(z_{i}),  j'(z_{i}), \tau_{1},z_{i})$.
    
    Finally, if we apply once more $\partial$ to equations (\ref{eq:partial1}) and (\ref{eq:partial2}), we obtain that $j''(\tau_{i})$ is algebraic over $\mathbb{Q}(j(\tau_{1}), j(z_{i}), j'(\tau_{1}),  j'(z_{i}), j''(z_{i}), \tau_{1}, z_{i})$.
    \item If $z_{1}\in \mathrm{gcl}_{\mathbb{Q}(\overline{\tau})}(z_{2},\ldots,z_{n},C)$, then suppose that for some $g\in G(\mathbb{Q}(\overline{\tau}))$ and some $x\in\left\{z_{2},\ldots,z_{n}\right\}\cup C$ we have that $z_{1} = gx$. This equality can be translated into a polynomial relation of the form: $(cx+d)z_{1} = ax+b$ (where $a$, $b$, $c$, and $d$ are the usual parameters of $g$), meaning that $z_{1}$ is algebraic over $\mathbb{Q}(a,b,c,d,x)$. Given that $a$, $b$, $c$, and $d$ are in $\mathbb{Q}(\overline{\tau})$, we obtain the desired result. 
\end{enumerate}
\end{proof}

With this proposition we can now obtain the following theorem.

\begin{thm}
\label{thm:gdim}
Let $(K,D)$ be a $j$-field and $C\subseteq K$ be $j\mathrm{cl}$-closed. Let $\tau_{1},\ldots,\tau_{m}\in D$ be such that they are  $j\mathrm{cl}$-independent over $C$. Then for any $z_{1},\ldots,z_{n}\in D$:
\begin{equation*}
    \mathrm{t.d.}(j(\overline{z}),j'(\overline{z}),j''(\overline{z})/C,\overline{\tau},\overline{z}) + \dim_{\mathbb{Q}(\overline{\tau})}^{g}(\overline{z}/C) - 3\dim^{g}(\overline{z}/C)\geq 0.
\end{equation*}
\end{thm}
\begin{proof}
From Corollary \ref{cor:reform} we get:
\begin{equation*}
\mathrm{t.d.}(\overline{\tau},\overline{z},j(\overline{\tau}), j(\overline{z}), j'(\overline{\tau}), j'(\overline{z}), j''(\overline{\tau}), j''(\overline{z})/C) \geq 3\dim^{g}(\overline{\tau},\overline{z}/C) + m.
\end{equation*}
Expanding using the addition formula we get:
\begin{eqnarray*}
    &&\mathrm{t.d.}(\overline{\tau}/C) + \mathrm{t.d.}(\overline{z}/C,\overline{\tau})+ \mathrm{t.d.}(j(\overline{z}),j'(\overline{z}),j''(\overline{z})/C,\overline{\tau},\overline{z})\\
    &&\qquad +\mathrm{t.d.}(j(\overline{\tau}), j'(\overline{\tau}), j''(\overline{\tau}) /C,\overline{\tau},\overline{z}, j(\overline{z}), j'(\overline{z}),j''(\overline{z})) 
     - 3\dim^{g}(\overline{\tau}/C,\overline{z})\\
     && \qquad- 3\dim^{g}(\overline{z}/C)\geq m.
\end{eqnarray*}
Using that $\mathrm{t.d.}(\overline{\tau}/C) = m$ and both inequalities from Proposition \ref{prop:moreresults}, we obtain the statement of the theorem.
\end{proof}

\begin{cor}
\label{cor:ineq}
Let $(K,D)$ be a $j$-field and $C\subseteq K$ be $j\mathrm{cl}$-closed. Let $\tau_{1},\ldots,\tau_{m}\in D$ be such that they are  $j\mathrm{cl}$-independent over $C$. Let $A$ be a subset of $C$. Then for any $z_{1},\ldots,z_{n}\in D$: 
\begin{equation*}
    \mathrm{t.d.}(j(\overline{z}),j'(\overline{z}),j''(\overline{z})/C,\overline{\tau},\overline{z}) + 3\dim_{\mathbb{Q}(\overline{\tau})}^{g}(\overline{z}/A) - 3\dim^{g}(\overline{z}/A)\geq 2\dim^{g}_{\mathbb{Q}(\overline{\tau})}(\overline{z}/C).
\end{equation*}
\end{cor}
\begin{proof}
By the previous theorem we get:
\begin{equation*}
    \mathrm{t.d.}(j(\overline{z}),j'(\overline{z}),j''(\overline{z})/C,\overline{\tau},\overline{z}) + \dim_{\mathbb{Q}(\overline{\tau})}^{g}(\overline{z}/C) - 3\dim^{g}(\overline{z}/C)\geq 0.
\end{equation*}
So we also get:
\begin{equation*}
    \mathrm{t.d.}(j(\overline{z}),j'(\overline{z}),j''(\overline{z})/C,\overline{\tau},\overline{z}) + 3\dim_{\mathbb{Q}(\overline{\tau})}^{g}(\overline{z}/C) - 3\dim^{g}(\overline{z}/C)\geq 2\dim^{g}_{\mathbb{Q}(\overline{\tau})}(\overline{z}/C).
\end{equation*}

If $\overline{\tau}$ is $j\mathrm{cl}$-independent over $C$, then it is algebraically independent over $C$. So, by Lemma \ref{lem:disjoint1}, we get that $\mathbb{Q}(\overline{\tau})$ is geodesically disjoint from $C$ over $\mathbb{Q}$. Now we can use Lemma \ref{lem:disjoint2}. 
\end{proof}

Now we prove our main theorem.

\begin{thm}[General Main Theorem]
\label{thm:main}
Let $(K,D)$ be a $j$-field. Let $\tau_{1},\ldots,\tau_{m},z_{1},\ldots,z_{n}\in D$ be such that $\overline{\tau}$ is $j\mathrm{cl}$-independent and $\overline{z}$ is $\mathrm{gcl}$-independent. Let $F = \mathbb{Q}(\overline{\tau})$. Suppose for each $1\leq i\leq m$, $g_{i}\in G^{\mathbb{Q}(\tau_{i})}$ is not of the form $ah$, where $a\in F$ and $h\in G$. Let $\partial_{1},\ldots,\partial_{m}$ be a corresponding system of $j$-derivations for $\overline{\tau}$ and set $C=\cap_{i}\ker\partial_{i}$. Suppose any of the following conditions is satisfied:
\begin{enumerate}[(a)]
\item The set $\left\{\bar{z}, g_{1}\bar{z},\ldots, g_{m}\bar{z}\right\}$ is $\mathrm{gcl}$-independent.
\item $\bar{z}\subseteq C$ and each $z_{i}$ is non-special.
\end{enumerate}
Then: 
\begin{equation*}
    \mathrm{t.d.}\left(j(\bar{z}),j'(\bar{z}),j''(\bar{z}),j(\bar{g}\bar{z}),j'(\bar{g}\bar{z}),j''(\bar{g}\bar{z})/F,C,\bar{z}\right)\geq 3nm.
\end{equation*}
\end{thm}
\begin{proof}
Using that $\dim_{F}^{g}(\bar{z},\bar{g}\bar{z}) = \dim^{g}_{F}(\bar{z})$ and Corollary \ref{cor:ineq} (with $A=\emptyset$), we get:
\begin{equation*}
    \mathrm{t.d.}\left(j(\bar{z}),j'(\bar{z}),j''(\bar{z}),j(\bar{g}\bar{z}),j'(\bar{g}\bar{z}),j''(\bar{g}\bar{z})/F,C,\bar{z}\right) \geq 3\dim^{g}(\bar{z},\bar{g}\bar{z})-  3\dim_{F}^{g}(\bar{z})+2\dim^{g}_{F}(\bar{z}/C).
\end{equation*}
\begin{enumerate}
\item The result follows immediately as
\begin{equation*}
    \mathrm{t.d.}\left(j(\bar{z}),j'(\bar{z}),j''(\bar{z}),j(\bar{g}\bar{z}),j'(\bar{g}\bar{z}),j''(\bar{g}\bar{z})/F,C,\bar{z}\right) \geq 3n(m+1) - 3n = 3nm.
\end{equation*}
\item In this case, we need to show that $\dim^{g}(\bar{z},\bar{g}\bar{z})\geq\dim^{g}_{F}(\bar{z})+n$. But using Lemma \ref{lem:disjoint1} we know that $F\bot^{g}_{\mathbb{Q}}C$. So $\dim^{g}_{F}(\bar{z})=n$, which means that, unless some $z_{i}$ is $F$-special, then the points $\bar{z},\bar{g}\bar{z}$ are $\mathrm{gcl}$-independent. However if any $z_{i}$ is $F$-special, then by Lemma \ref{lem:disjoint1}, they are $\mathbb{Q}$-special.
\end{enumerate} 
\end{proof}

There is a slight difference between this statement and that of Theorem \ref{thm:exp}. In the exponential case one only asks that $\exp(\overline{x})$ be multiplicatively independent, whereas in our theorem we ask (in the first case) that $\bar{z},\bar{g}\bar{z}$ be geodesically independent. This is because $j$ is not as ``nice'' as $\exp$. Suppose, for instance, that $n=m=1$ and that $z$ and $gz$ are in the same $G$-orbit (which would mean that $z$ is $F$-special). If $z\notin C$, then Corollary \ref{cor:ineq} can only give us the lower bound: 
\begin{equation*}
    \mathrm{t.d.}\left(j(z),j'(z),j''(z),j(gz),j'(gz),j''(gz)/F,C,z\right)\geq 2.
\end{equation*}
So the obstacle for proving a stronger result is in the $F$-special points. 

\subsection{Essential Counterexamples to MSC}
Now we relate our transcendence results to the modular Schanuel conjecture. 

\begin{defi}
Let $\overline{z}$ be a tuple in $D$ and let $B\subseteq K$. We define the \emph{predimension of $\overline{z}$ over $B$} to be:
\begin{equation*}
    \delta(\overline{z}/B) = \mathrm{t.d.}(\overline{z},j(\overline{z}),j'(\overline{z}),j''(\overline{z})/B,j(A),j'(A),j''(A)) - 3\dim^{g}(\overline{z}/B),
\end{equation*}
where $A= B\cap D$. When $B=\emptyset$ we write simply $\delta(\overline{z})$. By Proposition \ref{prop:as}, $\delta(\overline{z}/C)\geq\dim^{j}(\overline{z}/C)$.
\end{defi}

\begin{remark}
\label{rem:predim}
It is a simple exercise to see that for any pregeometry $\mathrm{cl}$ on a set $X$ we have that $\mathrm{cl}(A/B, C) = \mathrm{cl}(A, C/B) - \mathrm{cl}(C/B)$, where $A,B,C$ are any three subsets of $X$. Then, on a $j$-field $(K,D)$, for any finite $C\subseteq D$ we get that $\delta(\overline{z}/B, C) = \delta(\overline{z}, C/B) - \delta(C/B)$. So, if we choose $B=\emptyset$, we get that $\delta(A,C) = \delta(A/C)+\delta(C)$. Also note that for any finite $A\subseteq D$ we have that $\delta(A) = \delta(\mathrm{gcl}(A))$. 
\end{remark}

\begin{defi}
A tuple $\overline{z}$ from $D$ is called an \emph{essential counterexample} to the modular Schanuel conjecture if it does not satisfy inequality (\ref{eq:msc}) and for every tuple $\overline{c}$ from $\mathrm{gcl}(\overline{z})$ we have that $\delta(\overline{z})\leq\delta(\overline{c})$. 
\end{defi}

A counterexample of MSC is a tuple $\overline{z}$ of $D$ such that $\delta(\overline{z})<0$. On $\mathbb{C}$, using the relation given by the modular polynomials and their derivatives, we get that for any $z_{0}$ and any tuple $\overline{z}$ from $\mathbb{H}^{+}\cup\mathbb{H}^{-}$, $\delta(\overline{z},z_{0})\leq\delta(\overline{z})+1$. This means that if we had a counterexample $\overline{z}$ such that $\delta(\overline{z})<-1$, then we can produce $|\mathbb{C}|$ many counterexamples. Essential counterexamples can be used to classify families of counterexamples of MSC. Clearly, every counterexample of MSC has an essential counterexample in its $\mathrm{gcl}$-closure. 

\begin{prop}
 Let $(K,D)$ be a $j$-field. Let $\overline{z}\in D$ be such that for some $i$, $z_{i}\notin j\mathrm{cl}(\emptyset)$. Then $\overline{z}$ is not an essential counterexample. 
\end{prop}
\begin{proof}
Let $A = \mathrm{gcl}(\overline{z})\cap j\mathrm{cl}(\emptyset)$. By hypothesis, $A\neq \mathrm{gcl}(\overline{z})$. As $A\subseteq j\mathrm{cl}(\emptyset)$, then $j(A), j'(A), j''(A)\subseteq j\mathrm{cl}(\emptyset)$ by Proposition \ref{prop:jcl}, so:
\begin{equation*}
    \mathrm{t.d.}(\overline{z},j(\overline{z}),j'(\overline{z}),j''(\overline{z})/A, j(A), j'(A), j''(A))\geq \mathrm{t.d.}(\overline{z},j(\overline{z}),j'(\overline{z}),j''(\overline{z})/j\mathrm{cl}(\emptyset)).
\end{equation*}
Also $\dim^{g}(\overline{z}/A) = \dim^{g}(\overline{z}/j\mathrm{cl}(\emptyset))$, because $j\mathrm{cl}(\emptyset)$ is $\mathrm{gcl}$-closed. Using Proposition \ref{prop:as} we get:
\begin{equation*}
    \delta(\overline{z}/A)\geq\delta(\overline{z}/j\mathrm{cl}(\emptyset))\geq \dim^{j}(\overline{z}/j\mathrm{cl}(\emptyset))\geq 1.
\end{equation*}
Let $\overline{b}\in A$ be such that $\delta(\overline{z}/\overline{b}) = \delta(\overline{z}/A)$. Then $\delta(\overline{b}) = \delta(\overline{z},\overline{b}) - \delta(\overline{z}/\overline{b}) = \delta(\overline{z}) - \delta(\overline{z}/A) < \delta(\overline{z})$, so $\overline{z}$ is not an essential counterexample.
\end{proof}

\begin{cor}
By Remark \ref{rem:countable} we get that on $\mathbb{C}$ there are at most countably many essential counterexamples to the modular Schanuel conjecture. 
\end{cor}

\subsection{Final Remarks}

We have succeded in producing a transcendence result for the $j$ function in the spirit of the main theorem of \cite{bays-kirby-wilkie}. The key ingredient to make everything work is the Ax-Schanuel theorem for $j$, whereas the general framework can be adapted for any other modular function. 

For instance, consider the modular $\lambda$ function (see Chapter III.1 of \cite{silverman} for details). On $\mathbb{C}$, every elliptic curve is isomorphic to an elliptic curve in the Legendre form $E_{\lambda} : Y^{2} = X(X-1)(X-\lambda)$, where $\lambda\neq 0,1$. The $j$-invariant of such an elliptic curve is:
\begin{equation*}
    j(E_{\lambda}) = 2^{8}\frac{(\lambda^{2}-\lambda+1)^{3}}{\lambda^{2}(\lambda-1)^{2}}.
\end{equation*}
We can then define the function: $\lambda:\mathbb{H}\rightarrow\mathbb{C}\setminus\left\{0,1\right\}$ given by $\tau\mapsto j(E_{\lambda})$. This function is algebraic over $j$ and so we will have a family of polynomials for $\lambda$ like the modular polynomials for $j$, and with them we can define ``$\lambda$-fields''. We can also obtain an Ax-Schanuel type statement for $\lambda$, and then, with the method described here, we can obtain a generic transcendence result for $\lambda$ like Theorem \ref{thm:main}.

Considering the results for the exponential function of \cite{bays-kirby-wilkie} and our work for $j$, there seems to be a general philosophy about what is needed in order to repeat this strategy with a given holomorphic function $f$ to obtain a transcendence result like in Theorem \ref{thm:gdim}. Of course, one needs an Ax-Schanuel statement for $f$, i.e. an inequality of the form:
\begin{equation*}
    \mathrm{t.d.}(\overline{z},f(\overline{z}),\ldots,f^{(k)}(\overline{z}))\geq k\dim(\overline{z}) + \dim^{f}(\overline{z}).
\end{equation*}
Here $\dim^{f}$ is the dimension corresponding to the pregeometry defined by $f$-derivations (i.e. derivations that respect the derivatives of $f$ in the same sense as $j$-derivations). One can count on $f$-derivations defining a pregeometry, but if one wants to construct non-trivial $f$-derivations on $\mathbb{C}$, then one also has to ask that $f$ (or part of $f$, if $f$ has some symmetry properties) be definable in $\mathbb{R}_{\mathrm{an},\exp}$ (see \cite{jones-kirby-servi} and \cite{wilkie} for a general study of this topic). The dimension $\dim$ on the other hand is more interesting. This dimension has to be given by a pregeometry $\mathrm{cl}$ with (at least) the following characteristics: 
\begin{enumerate}[(i)]
    \item $\mathrm{cl}$ has to be definable in the field structure of $\mathbb{Q}$, so that we have a family of pregeometries for every subfield, like we have with $\mathrm{gcl}_{F}$ or with the pregeometry ``linear span''. 
    \item $\mathrm{cl}$ has to identify the algebraic relations between the functions $f$ and its derivatives (at least the derivatives considered in the Ax-Schanuel statement). 
\end{enumerate}

\subsection{Open Problems}
Even if we managed to obtain a transcendence result, there are still a couple of questions that remain open and that could be the subject of future work. 
\begin{enumerate}[(a)]
    \item Is the inequality of Theorem \ref{thm:main} sharp? The answer should be: no. How can one improve the inequality (short of proving MSC)? Considering MSC or the Modular ALW theorem, one would expect that there are no algebraic relations between $j$ and its derivatives evaluated in generic points of different $G$-orbits. A sharp statement might then be obtained by replacing $3nm$ by $4nm$. 
    \item Is there a more explicit description of the pregeometry $j\mathrm{cl}$? In the case of exponential fields it turns out that the pregeometry defined by exponential derivations agrees with the pregeometry defined by Khovanskii systems (this is defined by a non-degenerate system of equations of exponential polynomials). In fact, it is also known that these pregeometries define a dimension that agrees with the one coming from the corresponding predimension. For these results see \cite{kirby}. A natural generalisation of Khovanskii systems for $j$-fields can be defined, and one can ask if they define the pregeometry $j\mathrm{cl}$. However, this is not yet known as the strategy used for exponential fields relies too much on the properties of the exponential map (like the fact that it is a group homomorphism, and the shape of its differential equation). New ideas are needed here to continue this project. An equivalent formulation of this problem was already mentioned, in its general form, at the end of section 4 of \cite{jones-kirby-servi}. To exemplify how different the proofs of transcendence statements of $j$ can differ from analogous statements for $\exp$, it is enough to compare the proofs of the Ax-Schanuel theorem for $j$ and Ax's proof of Theorem $3$ of \cite{ax}. 
\end{enumerate}

\subsection*{Acknowledgements}
I would like to thank Jonathan Pila for invaluable guidance and supervision. I would also like to thank Vahagn Aslanyan and Alex Wilkie for fruitful discussions around the topics presented here. Finally, I would like to thank the referees for their many useful comments.

\end{document}